\documentclass[12pt]{amsart}
\usepackage[english]{babel}
\usepackage[pdftex,paper=a4paper,portrait=true,textwidth=450pt,textheight=675pt,tmargin=3cm,marginratio=1:1]{geometry}
\usepackage{amsfonts}
\usepackage[dvips]{graphics}
\usepackage[colorinlistoftodos]{todonotes}
\usepackage{amsmath}
\usepackage{amsthm}
\usepackage{amssymb}
\usepackage{bbm}
\usepackage{cancel}
\usepackage{color}
\usepackage[curve]{xypic}
\usepackage{graphicx}
\usepackage{verbatim}
\usepackage{version}
\usepackage{color}

\newtheorem{theorem}{Theorem}[section]
\newtheorem{corollary}[theorem]{Corollary}
\newtheorem{proposition}[theorem]{Proposition}

\newtheorem{conjecture}[theorem]{Conjecture}

\theoremstyle{definition}
\newtheorem{definition}[theorem]{Definition}

\newtheorem{remark}[theorem]{Remark}

\theoremstyle{property}

\DeclareFontFamily{OT1}{rsfs}{}
\DeclareFontShape{OT1}{rsfs}{n}{it}{<-> rsfs10}{}
\DeclareMathAlphabet{\curly}{OT1}{rsfs}{n}{it}

\newcommand\sfT{\mathsf T}

\newcommand\SW{\mathrm{SW}}

\newcommand\Eu{\mathrm{Eu}}
\newcommand\SL{\mathrm{SL}}

\renewcommand\O{\mathcal O}
\newcommand\PP{\mathbb P}

\newcommand\bfv{{\mathbf v}}

\newcommand\mdot{{\scriptscriptstyle\bullet}}

\newcommand\cE{\mathcal E}

\newcommand\cA{\mathcal A}

\newcommand\E{\mathbb E}

\newcommand\Coeff{\mathrm{Coeff}}
\newcommand\pt{\mathrm{pt}}
\newcommand\vir{\mathrm{vir}}
\newcommand\odd{\mathrm{odd}}

\newcommand\td{\mathrm{td}}
\newcommand\bslambda{\boldsymbol \lambda}

\newcommand\bsmu{\boldsymbol \mu}
\newcommand\C{\mathbb C}

\newcommand\FF{\mathbb F}

\newcommand\sfZ{\mathsf Z}
\newcommand\sfA{\mathsf A}
\newcommand\sfL{\mathsf L}

\newcommand\Q{\mathbb Q}

\newcommand\R{\mathbb R}

\newcommand\Z{\mathbb Z}
\newcommand\cZ{\mathcal Z}

\newcommand\s{\mathfrak s}
\newcommand\I{\mathcal I}

\newcommand\INTO{\ar@{^{(}->}[r]}

\newcommand\ch{\operatorname{ch}}
\newcommand\ev{\operatorname{ev}}
\newcommand\vd{\operatorname{vd}}
\newcommand\cob{\operatorname{cob}}
\newcommand\elg{\operatorname{ell}}

\renewcommand\hom{\mathcal{H}{\it{om}}}
\newcommand\Hom{\operatorname{Hom}}

\newcommand\Pic{\operatorname{Pic}}

\newcommand\Spec{\operatorname{Spec}\,}

\newcommand\Sym{\operatorname{Sym}}
\newcommand\beq[1]{\begin{equation}\label{#1}}
\newcommand\eeq{\end{equation}}
\newcommand\beqa{\begin{eqnarray*}}
\newcommand\eeqa{\end{eqnarray*}}

\begin{document}
\title[A rank 2 DMVV formula]{A rank 2 Dijkgraaf-Moore-Verlinde-Verlinde formula}

\author[L.~G\"ottsche and M.~Kool]{Lothar G\"ottsche and Martijn Kool}
\maketitle

\begin{abstract}
We conjecture a formula for the virtual elliptic genera of moduli spaces of rank 2 sheaves on minimal surfaces $S$ of general type. We express our conjecture in terms of the Igusa cusp form $\chi_{10}$ and Borcherds type lifts of three quasi-Jacobi forms which are all related to the Weierstrass elliptic function. We also conjecture that the generating function of virtual cobordism classes of these moduli spaces depends only on $\chi(\O_S)$ and $K_S^2$ via two universal functions, one of which is determined by the cobordism classes of Hilbert schemes of points on $K3$. We present generalizations of these conjectures, e.g.~to arbitrary surfaces with $p_g>0$ and $b_1=0$.

We use a result of J.~Shen to express the virtual cobordism class in terms of descendent Donaldson invariants. In a prequel we used T.~Mochizuki's formula, universality, and toric calculations to compute such Donaldson invariants in the setting of virtual $\chi_y$-genera. Similar techniques allow us to verify our new conjectures in many cases.
\end{abstract}

\section{Introduction}  \label{intro}

Denote by $S$ a smooth projective complex surface with $b_1(S) = 0$. For a polarization $H$ on $S$, denote by $M:=M_{S}^{H}(r,c_1,c_2)$ the moduli space of rank $r$ Gieseker $H$-stable torsion free sheaves on $S$ with Chern classes $c_1 \in H^2(S,\Z)$ and $c_2 \in H^4(S,\Z)$. Suppose there are no rank $r$ strictly Gieseker $H$-semistable sheaves with Chern classes $c_1, c_2$. Then $M_{S}^{H}(r,c_1,c_2)$ is projective. Moreover it has a  perfect obstruction theory, which was studied by T.~Mochizuki in his theory of algebraic Donaldson invariants \cite{Moc}. The virtual tangent bundle $T^{\vir}$ of $M$ is given by 
\begin{equation*} 
T^\vir = R\pi_{M*} R\hom(\E,\E)_0[1],
\end{equation*}
where $\E$ denotes a universal sheaf on $M \times S$, $\pi_M : M \times S \rightarrow M$ is projection, and $(\cdot)_0$ denotes the trace-free part.\footnote{Although $\E$ may only exist \'etale locally, $R\pi_{M*} R\hom(\E,\E)_0$ exists globally \cite[Sect.~10.2]{HL}.} 
Therefore we have a virtual cycle $[M]^{\vir}$ of degree 
\begin{equation} \label{vdformula}
\vd(M)=2rc_2-(r-1)c_1^2-(r^2-1)\chi(\O_S).
\end{equation}

The first topic of this paper is the virtual elliptic genera $Ell^\vir(M)$. Virtual elliptic genera of schemes with perfect obstruction theory were introduced by the first author and B.~Fantechi in \cite{FG}. They are defined as follows. Denote by $K^0(M)$ the $K$-group generated by locally free sheaves on $M$. For any rank $r$ vector bundle $V$ on $M$
\begin{align*}
\Lambda_t V := \sum_{n=0}^{r} [\Lambda^n V] \, t^n, \quad \Sym_t V := \sum_{n=0}^{\infty} [\Sym^n V] \, t^n.
\end{align*}
These definitions extend to complexes in $K^0(M)$ by setting $\Lambda_t(-V)=\Sym_{-t}V$ and $\Sym_t(-V)=\Lambda_{-t}V$.
The virtual elliptic genus, which refines the complex elliptic genus \cite{Hir, Wit, Kri}, is defined as
\begin{align*}
Ell^\vir(M) := y^{-\frac{\vd(M)}{2}} \chi_{-y}^{\vir}(M, \cE(T^\vir)).
\end{align*}
Here
\begin{align}
\begin{split} \label{cE}
\chi_{-y}^{\vir}(M,V) &:= \sum_{p \geq 0} (-y)^p \chi^\vir(M,V \otimes \Omega_{M}^{p,\vir}), \\
\cE(V) &:= \bigotimes_{n=1}^{\infty} \Lambda_{-yq^n} V^{\vee} \otimes \Lambda_{-y^{-1}q^n} V \otimes \Sym_{q^n} (V \oplus V^{\vee}),
\end{split}
\end{align}
and $\Omega_{M}^{p, \vir} := \Lambda^p T^{\vir \vee}$. See \cite{FG} for the definition of $\chi^\vir(M,\cdot)$.
 
In order to formulate our first conjecture, we recall notions from the theory of modular forms. See \cite{Kaw, EZ, Lib} and references therein. \\

\noindent \textbf{Siegel modular forms.} The ring of modular forms for $\SL(2,\Z)$ is $\C[G_4,G_6]$. For any even $k \geq 2$, $G_{k}$ denotes the Eisenstein series of weight $k$. It has the following expansion in $q=e^{2\pi i \tau}$
$$
G_{k}(q) = -\frac{B_{k}}{2k} + \sum_{n=1}^{\infty} \sigma_{k-1}(n) \, q^n, \ \sigma_k(n) := \sum_{d|n} d^k, \ B_k = k\textrm{th Bernoulli number}.
$$
 J.~Igusa found generators for the ring of Siegel modular forms of genus 2 \cite{Igu1, Igu2}. One such generator is the Igusa cusp form of weight 10 denoted by $\chi_{10}$, which is a function of
$$
\Omega = \left( \begin{array}{cc} \tau & z \\ z & \sigma \end{array} \right),
$$
which takes values in the (genus 2) Siegel upper half plane $\mathfrak{H}_2$. We set $i:=\sqrt{-1}$ and 
\begin{align*}
p := e^{2 \pi i \sigma}, \ q := e^{2\pi i \tau}, \ y := e^{2\pi i z}.
\end{align*}

\noindent \textbf{Weak Jacobi forms.} We recall that the ring of weak Jacobi forms of even weight and integer index is a polynomial algebra over $\C[G_4,G_6]$ with two generators $\phi_{-2,1}$, $\phi_{0,1}$ of respectively weight $-2$, $0$ and index $1$ \cite{EZ}. The generator $\phi_{-2,1}$ has Fourier expansion
$$
\phi_{-2,1}(q,y) = (y^{\frac{1}{2}} - y^{-\frac{1}{2}})^2 \prod_{n=1}^{\infty} \frac{(1-y q^n)^2(1-y^{-1}q^n)^2}{(1-q^n)^4}.
$$
The elliptic genus of an even $d$-dimensional Calabi-Yau manifold is a weak Jacobi form of weight 0 and index $\frac{d}{2}$ \cite{BL1, KYY, KM}. Specifically for a $K3$ surface
\begin{equation*} 
Ell(K3) = 2 \phi_{0,1}(q,y).
\end{equation*}

\noindent \textbf{Borcherds type lift.} For a meromorphic function $f : \mathfrak{H} \times \C \rightarrow \C$, where $\mathfrak{H}$ denotes the upper half plane, and with Fourier expansion
$$
f(q,y) = \sum_{m \geq 0, n \in \Z} c_{m,n} q^m y^n,
$$
we define its \emph{Borcherds type lift} by
\begin{align}
\begin{split} \label{L_a}
\sfL_a(f) :=& \, \exp\Big(-\sum_{l=1}^{\infty} p^{a l} (f|_{0,1}V_l)(\tau,z)\Big) \\
=& \prod_{l>0,m\geq 0,n \in \Z} (1-p^{al} q^m y^n)^{c_{lm,n}},
\end{split}
\end{align}
where $a \in \Z$ and $V_l$ are the Hecke operators in \cite[Sect.~I.4]{EZ}. We set $\sfL(f):=\sfL_1(f)$. We will also encounter Borcherds type lifts of
$$
f^{\ev}(q,y) := \sum_{m \geq 0, n \in \Z} c_{2m,n} q^{2m} y^{n}.
$$
We use the same definitions when $y$ can have half-integer powers.

A result by V.~Gritsenko and V.~Nikulin \cite{GN} based on the Borcherds lifting procedure \cite{Bor} expresses $\chi_{10}$ as an infinite product
\begin{equation} \label{GN}
\chi_{10}(p,q,y) = p \, \Delta(q) \, \phi_{-2,1}(q,y)  \,  \sfL(Ell(K3)),
\end{equation}
where 
$$
\Delta(q) = q \prod_{n=1}^{\infty} (1-q^n)^{24}
$$ 
is the discriminant modular form.

A celebrated formula from string theory by R.~Dijkgraaf, G.~Moore, E.~Verlinde, and H.~Verlinde \cite{DMVV} expresses the generating function for elliptic genera of Hilbert schemes $K3^{[n]}$ of $n$ points on a $K3$ surface as follows
\begin{equation*} 
\sum_{n=0}^{\infty} Ell(K3^{[n]}) \, p^n = \frac{1}{\sfL(Ell(K3))}.
\end{equation*}
This was proved in mathematics by L.~Borisov and A.~Libgober \cite{BL2, BL3}. In fact, they showed that the formula holds with $K3$ replaced by any smooth projective surface. By \eqref{GN}, this can be expressed as
$$
\sum_{n=0}^{\infty} Ell(K3^{[n]}) \, p^n =  \frac{p \, \Delta(q) \, \phi_{-2,1}(q,y)}{\chi_{10}(p,q,y)}.
$$

\noindent  \textbf{Quasi-Jacobi forms.} As mentioned above, elliptic genera of Calabi-Yau varieties are weak Jacobi forms (we only discussed the even-dimensional case). Libgober \cite{Lib} shows that the Calabi-Yau condition can be dropped as long as we relax the modularity property too. More precisely, elliptic genera of $d$-dimensional complex manifolds span a specific subspace in the ring of so-called \emph{quasi-Jacobi forms}. In our first conjecture, we encounter lifts of quasi- and weak Jacobi forms build from the following Jacobi-Eisenstein series:
\begin{align*}
G_{1,0}(q,y)&:=-\frac{1}{2} \frac{y + 1}{y - 1}+\sum_{n=1}^{\infty} \sum_{d | n} (y^d-y^{-d}) q^n, \\
G_{k,0}(q,y)&:=\Big(y\frac{\partial}{\partial y}\Big)^{k-1}G_{1,0}(q,y).
\end{align*}
Here $G_{1,0}$, $G_{2,0}$ are quasi-Jacobi forms of respectively weight $1$, $2$ and index $0$. Moreover $G_{2,0}-2G_2$, $G_{k,0}$ for $k>2$ are weak Jacobi forms of respectively weight $2$, $k$ and index 0.
%This sentence is a statement by Lothar that I did not check.
These functions can all be expressed in terms of the Weierstrass elliptic function 
$$
\wp(q,y) = \frac{1}{12} + \frac{y}{(1-y)^2} + \sum_{n=1}^{\infty} \sum_{d|n} d (y^d - 2 + y^{-d}) q^n=G_{2,0}(q,y) - 2G_2(q).
$$
By \cite[Thm.~3.6, 9.3]{EZ}, the weak Jacobi form $\phi_{0,1}$ can be written as
\begin{equation*} \label{phi2phi01}
\phi_{0,1}(q,y) = 12 (G_{2,0}(q,y) - 2G_2(q)) \phi_{-2,1}(q,y).
\end{equation*}
In a similar way we define
\begin{align*}
\phi_{0,\frac{k}{2}}(q,y):=\,&G_{k,0}(q,y) \phi_{-2,1}(q,y)^{\frac{k}{2}}, \quad k\ne 2.
\end{align*}
Then $\phi_{0,\frac{1}{2}}$ is a quasi-Jacobi form of weight $0$ and index $\frac{1}{2}$, and the $\phi_{0,\frac{k}{2}}$ for $k\ge 2$  are weak Jacobi forms of weight $0$  and index $\frac{k}{2}$. 
%This sentence is a statement by Lothar that I did not check.

Denote by $\SW(a)$ the Seiberg-Witten invariant of $S$ in class $a \in H^2(S,\Z)$.\footnote{We use Mochizuki's notation \cite{Moc}: $\SW(a)$ stands for $\widetilde{\SW}(2a-K_S)$ with $\widetilde{\SW}(b)$ being the usual Seiberg-Witten invariant in class $b \in H^2(S,\Z)$.} Then $a \in H^2(S,\Z)$ is called a \emph{Seiberg-Witten basic class} when $\SW(a) \neq 0$. Many surfaces have $0$ and $K_S$ as their only Seiberg-Witten basic classes, e.g.~minimal general type surfaces \cite[Thm.~7.4.1]{Mor}. We conjecture the following.\footnote{In Remark \ref{motivate} of Section \ref{conseqsec}, we motivate how we initially found the formula of Conjecture \ref{conjell}.}

\begin{conjecture} \label{conjell}
Let $S$ be a smooth projective surface such that $b_1(S) = 0$, $p_g(S)>0$, $K_S \neq 0$, and the only Seiberg-Witten basic classes of $S$ are $0$ and $K_S$. Let $H,c_1,c_2$ be chosen such that there are no rank 2 strictly Gieseker $H$-semistable sheaves on $S$ with Chern classes $c_1,c_2$. Let $M:=M_{S}^{H}(2,c_1,c_2)$. Then $Ell^{\vir}(M)$ is given by the coefficient of $p^{\vd(M)}$ of 
$$
\psi_S(p,q,y) := 8 \Bigg( \frac{1}{2\sfL_2(\phi_{0,1})} \Bigg)^{\chi(\O_S)}
\Bigg( \frac{2 \sfL_4(2\phi_{0,\frac{1}{2}} \phi_{0,\frac{3}{2}}) \sfL(-2 \phi_{0,\frac{1}{2}})}{\sfL_2\big(-2 \phi_{0,\frac{1}{2}}^{\ev}|_{(q^{\frac{1}{2}},y)}-\phi_{0,\frac{1}{2}}|_{(q^2,y^2)}+2\phi_{0,\frac{1}{2}}^2 \big) } \Bigg)^{K_S^2},
$$
where 
$$
\sfL_2(\phi_{0,1}) = \sfL_2(Ell(K3))^{\frac{1}{2}} = \Bigg( \frac{\chi_{10}(p^2,q,y)}{p^2 \, \Delta(q) \, \phi_{-2,1}(q,y)} \Bigg)^{\frac{1}{2}}. 
$$
\end{conjecture}
 
Next we shift our attention to algebraic cobordism theory \cite{LM, LP}. We consider the algebraic cobordism ring over a point with rational coefficients 
$$
\Omega_* := \bigoplus_{d=0}^{\infty} \Omega_d(\pt) \otimes_\Z \Q.
$$
This is isomorphic to the polynomial ring freely generated by the cobordism classes of $\PP^d$. Moreover, $\Omega_d(\pt) \otimes_\Z \Q$ has a basis
$$
v^I := v_1^{i_1} \cdots v_d^{i_d}, \ \mathrm{where} \ I = (i_1, \ldots, i_d) \in \Z_{\geq 0}^d \ \mathrm{and} \ |I|=\sum k i_k = d
$$
such that the cobordism class $[X]$ of a $d$-dimensional smooth projective variety $X$ is 
$$
[X] = \int_X \prod_{i=1}^{d}\big(1+ \sum_{k=1}^{\infty} x_i^k v_k\big),
$$
where $x_1, \ldots, x_d$ denote the Chern roots of $T_X$. The coefficients of $v^I$ appearing in this expression are symmetric in the Chern roots. From this it follows that the class $[X]$ is uniquely determined by its collection of \emph{Chern numbers}. 

The cobordism class of the Hilbert scheme $S^{[n]}$ of $n$ points on a smooth projective surface $S$ was studied by the first author, G.~Ellingsrud, and M.~Lehn \cite{EGL}. It was shown that there exist two universal functions $F_0, F_1 \in 1+\Q[v_1,v_2,\ldots][\![p]\!]$ such that 
$$
\sum_{n=0}^{\infty} [S^{[n]}] \, p^{n} = F_0^{\chi(\O_S)} F_1^{K_S^2},
$$
for any surface $S$. Therefore $F_0^2$ is the generating series of cobordism classes of $K3^{[n]}$.

When $M$ is a projective scheme with a perfect obstruction theory, J.~Shen \cite{She} constructed its virtual cobordism class 
$$
[M]^{\vir}_{\Omega_*} \in \Omega_{\vd}(M),
$$
where $\vd = \vd(M)$ is the virtual dimension (see also \cite{CFK} and \cite{LS} in the context of dg-manifolds/derived schemes). Denoting projection by $\pi : M \rightarrow \pt$, Shen proved that $\pi_* [M]^{\vir}_{\Omega_*}$ is uniquely determined by its collection of \emph{virtual Chern numbers}. In terms of the basis $v^I$, this can be expressed as follows. Let $T^\vir = [E_0 \rightarrow E_1]$ be a resolution by vector bundles and denote the Chern roots of $E_0$ by $x_1, \ldots, x_n$ and the Chern roots of $E_1$ by $u_1, \ldots, u_m$. Then
$$
\pi_* [M]^{\vir}_{\Omega_*} = \int_{[M]^\vir} \prod_{i=1}^{n}\big(1+ \sum_{k=1}^{\infty} x_i^k v_k\big) \prod_{j=1}^{m} \frac{1}{\big(1+ \sum_{k=1}^{\infty} u_j^k v_k\big)}.
$$
%\begin{MK} I did not prove there is a change of basis between virtual Chern classes (as in Shen) and this basis. I only wrote down the change of basis matrix for a few cases of low virtual dimension.\end{MK} 
%\begin{LG} I think it is o.k. if there are complaints we can certainly add a proof.%\end{LG}
\begin{conjecture} \label{conjcob}
There exists a universal power series $L(p,\bfv) \in 1+\Q[v_1,v_2,\ldots][\![p]\!]$ with the following property. Let $S$ be a smooth projective surface such that $b_1(S) = 0$, $p_g(S)>0$, $K_S \neq 0$, and the only Seiberg-Witten basic classes of $S$ are $0$ and $K_S$. Let $H,c_1,c_2$ be chosen such that there are no rank 2 strictly Gieseker $H$-semistable sheaves on $S$ with Chern classes $c_1,c_2$. Then for $M:=M_{S}^{H}(2,c_1,c_2)$ we have that $\pi_* [M]^{\vir}_{\Omega_*}$ is given by the coefficient of $p^{\vd(M)}$ of 

$$
\psi_S(p,\bfv) := 8 \Bigg( \frac{1}{2} \Big(\sum_{n=0}^{\infty} [K3^{[n]}] \, p^{2n}\Big)^{\frac{1}{2}} \Bigg)^{\chi(\O_S)} \Bigg( 2L(p,\bfv) \Bigg)^{K_S^2}.
$$
\end{conjecture}

\begin{remark}
Using the virtual Hirzebruch-Riemann-Roch theorem of \cite{CFK, FG}, the elliptic genera $Ell^{\vir}(M)$ in Conjecture \ref{conjell} can be expressed in terms of $q,y$ and virtual Chern numbers of $M$. In particular, Conjecture \ref{conjcob} implies Conjecture \ref{conjell} \emph{except} for the explicit expression for the power series which is raised to the power $K_S^2$. Specializing $Ell^{\vir}(M)$ to $q = 0$ gives the virtual $\chi_y$-genus $\chi_{y}^{\vir}(M)$. We conjectured an explicit formula for these virtual $\chi_{y}$-genera in \cite{GK1}, which is implied by Conjecture \ref{conjell}.\footnote{This follows from a basic calculation using the definitions and the Jacobi triple product identity.} Specializing further to $y=1$ gives the virtual Euler characteristics $e^{\vir}(M)$. The formula for these virtual Euler characteristics coincides with \emph{part of} a formula of C.~Vafa and E.~Witten from the physics literature \cite{VW}. Their formula was one of the main motivations for \cite{GK1}. The \emph{full} Vafa-Witten formula is (conjecturally) explained by Y.~Tanaka and R.~P.~Thomas's recently introduced Vafa-Witten invariants, which contain the virtual Euler characteristics that we computed in \cite{GK1} as well as contributions from ``other components with non-zero Higgs field''. 
\end{remark}

\begin{remark}
In physics language \cite{HIV}, the generating function for $e^{\vir}(M)$ describes a topological twist of $N=4$ supersymmetric Yang-Mills on the 4-manifold $S$. Moreover, $\chi_y^{\vir}(M)$ describes a 5-dimensional theory on $S \times S^1$ compactified along the circle $S^1$, $Ell^{\vir}(M)$ describes a 6-dimensional theory on $S \times T$ compactified along the real torus $T$, and $\pi_*[M]^{\vir}_{\Omega_*}$ appears to describe a new infinite-dimensional version of these theories.
\end{remark}

In this paper we verify Conjectures \ref{conjell} and \ref{conjcob} in a large number of examples by computer calculations (Section \ref{versec}). Our strategy is similar to the one followed in \cite{GK1} for virtual $\chi_y$-genus, which in turn relies on ideas from \cite{GNY1, GNY3}. Specifically:
\begin{itemize}
\item Write $Ell^{\vir}(M)$ and $\pi_* [M]^{\vir}_{\Omega_*}$ in terms of descendent Donaldson invariants of $S$ using a theorem of J.~Shen \cite{She}.
\item Write descendent Donaldson invariants of $S$ in terms of Seiberg-Witten invariants and certain explicit integrals over $S^{[n_1]} \times S^{[n_2]}$ using Mochizuki's formula \cite{Moc}.
\item Use a universality argument to show that the integrals over $S^{[n_1]} \times S^{[n_2]}$ are governed by seven universal functions.
\item Note that these seven universal functions are determined on $S = \PP^2$ and $\PP^1 \times \PP^1$, where we calculate them up to some order. 
\end{itemize}

In Section \ref{conseqsec} we generalize Conjectures \ref{conjell} and \ref{conjcob} in two different directions:
\begin{itemize}
\item Conjecture \ref{numconj} can be seen as a statement purely about intersection numbers on Hilbert schemes of points. Together with a strong version of Mochizuki's formula, it implies Conjectures \ref{conjell} and \ref{conjcob}. It also implies a generalization of Conjectures \ref{conjell} and \ref{conjcob} to arbitrary blow-ups of surfaces $S$ satisfying $b_1(S)=0$, $p_g(S)>0$, $K_S\neq 0$, and the only Seiberg-Witten basic classes of $S$ are $0$ and $K_S$.
\item Conjecture \ref{generalsurfconj} is a generalization of Conjectures \ref{conjell} and \ref{conjcob} to \emph{arbitrary} surfaces $S$ satisfying $b_1(S) = 0$ and $p_g(S)>0$. It implies a blow-up formula. It also implies a formula for surfaces with canonical divisor with irreducible reduced connected components. The latter refines an equation from Vafa-Witten \cite[Eqn.~(5.45)]{VW}. Conjecture \ref{generalsurfconj} itself can be seen as a refinement of (part of) a formula from the physics literature due to Dijkgraaf-Park-Schroers \cite[Eqn.~(6.1), lines 2+3]{DPS}.
\end{itemize}
Conjectures \ref{numconj} and \ref{generalsurfconj} are also checked in examples in Section \ref{versec}. 

The physics approach to the calculation of elliptic genera of instanton moduli spaces was discussed in N.~Nekrasov's PhD thesis \cite{Nek} and the papers \cite{LNS, BLN}. 

Some results in this paper, mostly the consequences discussed in Section \ref{conseqsec}, follow from minor modifications of arguments appearing in \cite{GK1} for the case of virtual $\chi_y$-genus. In order to keep this paper self-contained, we nevertheless reproduce the main idea of the proofs of these results (besides giving a reference to the full argument in \cite{GK1}). \\

\noindent \textbf{Acknowledgements.} 
We thank H.~Nakajima and J.~Shen for useful conversations.

\section{Notation} \label{notsec}

Throughout this paper, we deal with virtual cobordism classes and virtual elliptic genera simultaneously. Therefore we introduce the following notation. Using the functions appearing in Conjectures \ref{conjell} and \ref{conjcob} we define
\begin{align*}
F_0^{\cob}(p,\bfv) &:= \Big( \sum_{n=0}^{\infty} [K3^{[n]}] \, p^{2n} \Big)^{\frac{1}{2}}, \qquad F_1^{\cob}(p,\bfv) := L(p,\bfv), \\
F_0^{\elg}(p,q,y) &:=  \frac{1}{\sfL_2(\phi_{0,1})}, \qquad F_1^{\elg}(p,q,y) :=  \frac{\sfL_4(2\phi_{0,\frac{1}{2}} \phi_{0,\frac{3}{2}}) \sfL(-2 \phi_{0,\frac{1}{2}})}{\sfL_2\big(-2 \phi_{0,\frac{1}{2}}^{\ev}|_{(q^{\frac{1}{2}},y)}-\phi_{0,\frac{1}{2}}|_{(q^2,y^2)}+2\phi_{0,\frac{1}{2}}^2 \big) }.
\end{align*}
We also define
\begin{align*}
u^{\cob} := \bfv, \quad u^{\elg} := (q,y)
\end{align*}
and we view $\bfv = (v_1,v_2,\ldots )$ as formal variables. This notation allows us to consider both cases $F_i^{\star}(p,u^\star)$, $\star \in \{\cob,\elg\}$ simultaneously.

\section{Expression in terms of Donaldson invariants} 

We fix a smooth projective surface $S$ with $b_1(S)=0$ and polarization $H$. Denote by $M:=M_{S}^{H}(r,c_1,c_2)$ the moduli space of rank $r$ Gieseker $H$-stable torsion free sheaves on $S$ with Chern classes $c_1,c_2$. We assume there are no rank $r$ strictly Gieseker $H$-semistable sheaves on $S$ with Chern classes $c_1,c_2$, so $M_{S}^{H}(r,c_1,c_2)$ is projective. We also assume there exists a universal sheaf $\E$ on $M \times S$ (though we get rid of this assumption later in Remark \ref{univexists}). In this section, we use a result of Shen \cite{She} to show that 
$$
\pi_* [M]^{\vir}_{\Omega_*}, \ Ell^{\vir}(M)
$$
can be written in terms of descendent Donaldson invariants.

We first recall descendent insertions. Let $\sigma \in H^*(S,\Q)$ and $\alpha \geq 0$, then 
$$
\tau_{\alpha}(\sigma) := \pi_{M*} \big( \ch_{\alpha+2}(\E) \cap \pi_{S}^{*} \, \sigma \big),
$$
where $\pi_M : M \times S \rightarrow M$ and $\pi_S : M \times S \rightarrow S$ denote projections. We refer to $\tau_{\alpha}(\sigma)$ as a descendent insertion of degree $\alpha$. The insertions $\tau_0(\sigma)$ are called primary insertions. Donaldson invariants with primary insertions feature in the Witten conjecture, which is proved for algebraic surfaces in \cite{GNY3}. In this paper and \cite{GK1} we need higher descendents.

We introduce some further notation. For $X$ a projective $\C$-scheme and $E$ a rank $r$ vector bundle on $X$ with Chern roots $x_1, \dots, x_r$, we define
\begin{equation} \label{defTcob}
\sfT^{\cob}(E,\bfv) := \prod_{i=1}^{r}\big(1+ \sum_{k=1}^{\infty} x_i^k v_k\big),
\end{equation}
where we view $\bfv=(v_1, v_2,\ldots )$ as formal variables. By setting $\sfT^{\cob}(-E,\bfv) = 1/ \sfT^{\cob}(E,\bfv)$, we extend this definition to the entire $K$-group. Furthermore we introduce
\begin{equation} \label{defTelg}
\sfT^{\elg}(E,q,y) := y^{-\frac{r}{2}} \ch(\cE(E)) \, \ch(\Lambda_{-y} E^{\vee}) \, \td(E),
\end{equation}
where $\cE(\cdot)$ was defined in \eqref{cE}. The following multiplicative property plays an important role in Section \ref{univsec} 
\begin{align}
\begin{split} \label{mult}
\sfT^\star(E_1 + E_2,u^\star) &= \sfT^\star(E_1,u^\star) \, \sfT^\star(E_2,u^\star), 
\end{split}
\end{align}
for all $E_1, E_2 \in K^0(X)$ and $\star \in \{\cob,\elg\}$. Here we use the notation of Section \ref{notsec}. For $\star = \cob$ the multiplicative property is obvious and for $\star=\elg$, we use
$$
\cE(E_1 \oplus E_2) = \cE(E_1) \otimes \cE(E_2).
$$
In \cite{GK1} we introduced a similar expression, denoted by $\sfT_y(E,t)$, to deal with virtual $\chi_{y}$-genus and virtual Euler characteristic. The object $\sfT_y(E,t)$ also satisfies \eqref{mult}.

\begin{proposition} \label{desc}
Let $S,H,r,c_1,c_2$ and $M:= M_{S}^{H}(r,c_1,c_2)$ be as above. For both $\star \in \{\cob,\elg\}$, there exists a formal power series expression $P^{\star}(\E,u^\star)$ in variables $u^\star$ whose coefficients are polynomial expressions in certain descendent insertions $\tau_{\alpha}(\sigma)$ satisfying
\begin{align*}
\pi_* [M]^{\vir}_{\Omega_*} = \int_{[M]^{\vir}} P^{\cob}(\E,\bfv), \qquad Ell^{\vir}(M) = \int_{[M]^{\vir}} P^{\elg}(\E,q,y).
\end{align*}
\end{proposition}
\begin{proof}
By the virtual Hirzebruch-Riemann-Roch formula \cite{CFK, FG}
$$
Ell^{\vir}(M) = y^{-\frac{\vd(M)}{2}} \int_{[M]^{\vir}} \ch(\cE(T^\vir)) \, \ch(\Lambda_{-y}T^{\vir \vee}) \, \td(T^\vir).
$$
We can expand the integrand in $q^n y^m$ and write each coefficient as a polynomial expression in $c_i(T{^\vir})$. Since $\pi_* [M]^{\vir}_{\Omega_*}$ involves \emph{all} virtual Chern numbers, the existence of $P^{\elg}(\E,q,y)$ follows from the existence of $P^{\cob}(\E,\bfv)$. The existence of the polynomial $P^{\cob}(\E,\bfv)$ was proved by Shen \cite[Thm.~0.2]{She}.\footnote{Shen works with Pandharipande-Thomas invariants on threefolds. The result \cite[Thm.~0.2]{She} holds in our setting by replacing $-R \pi_{P*} R\hom(\mathbb{I}^{\mdot},\mathbb{I}^{\mdot})_0$ by $-R \pi_{M*} R\hom(\E,\E)_0$.}
\end{proof}

\section{Expression in terms of Hilbert schemes} \label{exprinHilb}

We consider $M_{S}^{H}(2,c_1,c_2)$ as in the previous section. Proposition \ref{desc} expresses virtual cobordism class and virtual elliptic genus in terms of descendent Donaldson invariants. We now recall Mochizuki's formula \cite[Thm.~1.4.6]{Moc}, which allows us to write descendent Donaldson invariants in terms of integrals over Hilbert schemes of points on $S$ and Seiberg-Witten invariants of $S$. See Section \ref{intro} for our conventions on Seiberg-Witten invariants.

We denote the Hilbert scheme of $n$ points on $S$ by $S^{[n]}$. On $S^{[n_1]} \times S^{[n_2]} \times S$, we have pull-backs of the universal ideal sheaves $\mathcal{I}_1$, $\mathcal{I}_2$ from $S^{[n_1]} \times S$, $S^{[n_2]} \times S$, which we denote by the same symbol. Moreover, for any $L \in \Pic(S)$ and $i=1,2$, we consider the tautological vector bundles
$$
L^{[n_i]} := p_* q^* L,
$$
with $p : \cZ_i \rightarrow S^{[n_i]}$, $q : \cZ_i \rightarrow S$ projections from the universal subscheme $\cZ_i \subset S^{[n_i]} \times S$.

We endow $S^{[n_1]} \times S^{[n_2]}$ with a trivial $\C^*$ action. We denote the generator of the character group of $\C^*$ by $\s$. Correspondingly, we write $s$ for the generator of 
$$
H^*(B\C^*,\Q) = H^*_{\C^*}(\pt,\Q) \cong \Q[s].
$$

\begin{remark} \label{masterC}
Roughly speaking, Mochizuki derives his formula \cite[Thm.~1.4.6]{Moc} from a virtual wall-crossing formula on a master space $\mathbb{M}$. This master space comes equipped with a $\C^*$ action and the Hilbert schemes $S^{[n_1]} \times S^{[n_2]}$  arise as components of the $\C^*$-fixed locus of $\mathbb{M}$. Although this master space itself plays no role in the formulation of Mochzuki's formula, Theorem \ref{mocthm} below, the trivial $\C^*$ action on $S^{[n_1]} \times S^{[n_2]}$ features prominently. We denote it by $\C^* = \C^*_{\mathbb{M}}$ when we want to stress its origin.
\end{remark}

Let $P(\E)$ be any polynomial in descendent insertions $\tau_{\alpha}(\sigma)$, which arises from a polynomial in Chern numbers of $T^\vir$ (e.g.~such as in Proposition \ref{desc}). Denote the group of Weil divisors on $S$ modulo linear equivalence by $A^1(S)$. For any $a_1, a_2 \in A^1(S)$ and $n_1, n_2 \in \Z_{\geq 0}$, we define (after Mochizuki) 
\begin{equation} \label{Psi}
\Psi(a_1,a_2,n_1,n_2) := \Coeff_{s^0} \Bigg( \frac{P(\I_1(a_1) \otimes \s^{-1} \oplus \I_2(a_2) \otimes \s)}{Q(\I_1(a_1) \otimes \s^{-1}, \I_2(a_2) \otimes \s)} \frac{\Eu(\O(a_1)^{[n_1]}) \, \Eu(\O(a_2)^{[n_2]} \otimes \s^2)}{(2s)^{n_1+n_2 - \chi(\O_S)}} \Bigg).
\end{equation}
We explain the notation appearing in this formula. Here $\I_i(a_i)$ is short-hand for $\I_i \otimes \pi_S^* \O(a_i)$, which are considered as sheaves on $S^{[n_1]} \times S^{[n_2]} \times S$. Similarly $\O(a_i)^{[n_i]}$ are considered as vector bundles on $S^{[n_1]} \times S^{[n_2]}$ pulled back from its factors. The scheme $S^{[n_1]} \times S^{[n_2]}$ has trivial $\C^*_{\mathbb{M}} = \C^*$ action, and we consider $\O(a_i)^{[n_i]}$ endowed with the \emph{trivial} $\C^*$-equivariant structure. Furthermore
$$
\O(a_2)^{[n_2]} \otimes \s^2
$$
denotes $\O(a_2)^{[n_2]}$ with $\C^*$-equivariant structure given by tensoring with character $\s^2$. Similarly, we endow $S^{[n_1]} \times S^{[n_2]} \times S$ with trivial $\C^*$ action, endow $\I_i(a_i)$ with trivial $\C^*$-equivariant structure, and denote by
$$
\I_1(a_1) \otimes \s, \qquad \I_2(a_2) \otimes \s^{-1}
$$
the $\C^*$-equivariant sheaves obtained by tensoring with the characters $\s$ and $\s^{-1}$ respectively. Next, $\Eu(\cdot)$ denotes $\C^*$-equivariant Euler class (which is the ordinary Euler class in $\Eu(\O(a_1)^{[n_1]}$). Furthermore, $P(\cdot)$ is the expression obtained from $P(\E)$ by formally replacing $\E$ by $\cdot$. For any $\C^*$-equivariant sheaves $E_1$, $E_2$ on $S^{[n_1]} \times S^{[n_2]} \times S$ flat over $S^{[n_1]} \times S^{[n_2]}$
$$
Q(E_1,E_2) :=\Eu(- R\pi_* R\hom(E_1,E_2) -  R\pi_*  R\hom(E_2,E_1)), 
$$
where $\pi : S^{[n_1]} \times S^{[n_2]} \times S \rightarrow S^{[n_1]} \times S^{[n_2]}$ denotes projection. All pull-backs and push-forwards in $P(\cdot), Q(\cdot,\cdot)$ are $\C^*$-equivariant with respect to the trivial $\C^*$ actions on $S^{[n_1]} \times S^{[n_2]} \times S$, $S^{[n_1]} \times S^{[n_2]}$, and $S$, and the only non-trivial equivariant structures come from the characters $\s^{\pm 1}$. Finally $\Coeff_{s^0}(\cdot)$ takes the coefficient of $s^0$.\footnote{This differs from Mochizuki who uses $p_g(S)$ instead of $\chi(\O_S)$ and takes a residue. Our definition differs by a factor $2$ from Mochizuki's. Mochizuki works on the moduli stack of oriented sheaves which maps to $M$ via a degree $\frac{1}{2}:1$ \'etale morphism, which accounts for the difference.} We define 
$$
\widetilde{\Psi}(a_1,a_2,n_1,n_2,s)
$$
by expression \eqref{Psi} but \emph{without} $\Coeff_{s^0}(\cdot)$. Let $c_1,c_2$ be a choice of Chern classes. For any decomposition $c_1 = a_1 + a_2$, Mochizuki defines
\begin{equation} \label{cA}
\cA(a_1,a_2,c_2) := \sum_{n_1 + n_2 = c_2 - a_1 a_2} \int_{S^{[n_1]} \times S^{[n_2]}} \Psi(a_1,a_2,n_1,n_2).
\end{equation}
We denote by $\widetilde{\cA}(a_1,a_2,c_2,s)$ the same expression with $\Psi$ replaced by $\widetilde{\Psi}$. \\

With these ingredients, Mochizuki's formula is as follows \cite[Thm.~1.4.6]{Moc}:
\begin{theorem}[Mochizuki] \label{mocthm}
Let $S$ be a smooth projective surface with $b_1(S) = 0$ and $p_g(S) >0$. Let $H, c_1,c_2$ be chosen such that there exist no rank 2 strictly Gieseker $H$-semistable sheaves on $S$ with Chern classes $c_1,c_2$. Suppose the universal sheaf $\E$ exists on $M_{S}^{H}(2,c_1,c_2) \times S$. Suppose the following conditions hold:
\begin{enumerate}
\item[(i)] $\chi(\ch) > 0$, where $\chi(\ch) := \int_S \ch \cdot \td(S)$ and $\ch = (2,c_1,\frac{1}{2}c_1^2 - c_2)$.
\item[(ii)] $p_{\ch} > p_{K_S}$, where $p_{\ch}$ and $p_{K_S}$ are the reduced Hilbert polynomials of $\ch$ and $K_S$.
\item[(iii)] For all SW basic classes $a_1$ with $a_1 H \leq (c_1 -a_1) H$ the inequality is strict. 
\end{enumerate}
Let $P(\E)$ be any polynomial in descendent insertions, which arises from a polynomial in Chern numbers of $T^\vir$ (e.g.~such as in Prop.~\ref{desc}). Then
\begin{equation} \label{mocform}
\int_{[M_{S}^{H}(2,c_1,c_2)]^{\vir}} P(\E) = -2^{1-\chi(\ch)} \sum_{{\scriptsize{\begin{array}{c} c_1 = a_1 + a_2 \\ a_1 H < a_2 H \end{array}}}} \SW(a_1) \, \cA(a_1,a_2,c_2).
\end{equation}
\end{theorem}

\begin{remark} \label{univexists}
The assumption that $\E$ exists on $M \times S$, where $M:=M_{S}^{H}(2,c_1,c_2)$ is not needed. As remarked in the introduction, $T^{\vir} = -R\pi_{*}R\hom(\E,\E)_0$ always exists globally so the left-hand side of Mochizuki's formula always makes sense. Also, Mochizuki \cite{Moc} works over the Deligne-Mumford stack of oriented sheaves, which always has a universal sheaf. This can be used to show that global existence of $\E$ on $M \times S$ can be dropped from the assumptions. In fact, when working on the stack, $P$ can be \emph{any} polynomial in descendent insertions defined using the universal sheaf on the stack.
\end{remark}

\begin{remark} \label{assumpmocthm}
The first author, Nakajima, and K.~Yoshioka conjecture in \cite{GNY3} that assumptions (ii) and (iii) can be dropped from Theorem \ref{mocthm}. They also conjecture that the sum in Mochizuki's formula can be replaced by a sum over \emph{all} classes $a_1 \in H^2(S,\Z)$. Assumption (i) is necessary as we see from our calculations in Section \ref{versec}. 
\end{remark}

From Proposition \ref{desc} and Theorem \ref{mocthm} we deduce that
$$
\pi_* [M]^{\vir}_{\Omega_*}, \ Ell^{\vir}(M_{S}^{H}(2,c_1,c_2))
$$
are given by equation \eqref{mocform} by taking respectively $\star = \cob, \elg$ and
\begin{equation} \label{choiceP}
P(\E) = P^{\star}(\E,u^\star) = \sfT^\star(-R \pi_* R\hom(\E,\E)_0,u^\star) ,
\end{equation}
where $\E$ is replaced by
$$
\I_1(a_1) \otimes \s^{-1} \oplus \I_2(a_2) \otimes \s.
$$
We note that the rank of
\begin{equation}  \label{Erepl}
-R \pi_* R\hom(\I_1(a_1) \otimes \s^{-1} \oplus \I_2(a_2) \otimes \s,\I_1(a_1) \otimes \s^{-1} \oplus \I_2(a_2) \otimes \s)_0
\end{equation} 
equals the rank of $T^{\vir}$.

Recall that $S^{[n_1]} \times S^{[n_2]}$ has trivial $\C_{\mathbb{M}}^* = \C^*$ action and for any $\C^*$-equivariant complex $E$ on $S^{[n_1]} \times S^{[n_2]}$, such as the complex given by \eqref{Erepl}, we denote by $\sfT^\star_{\C^*}(E,u^\star)$ the $\C^*$-equivariant analog of \eqref{defTcob} and \eqref{defTelg}. For $\star = \elg$, this means replacing $\ch$, $\td$ by $\ch^{\C^*}$, $\td^{\C^*}$ in \eqref{defTelg}. For $\star = \cob$, this means replacing Chern roots by $\C^*$-equivariant Chern roots in \eqref{defTcob}. In particular, we do not need to work with a $\C^*$-equivariant cobordism ring and we regard $v_k$ as \emph{formal variables}. Only in the formula of Conjecture \ref{conjcob} these formal variables become the generators of the (non-equivariant) algebraic cobordism ring $\Omega_*$ as explained in the introduction.

\section{Expression in terms of universal functions} \label{univsec}

Let $S$ be any smooth projective surface. We are interested in expression \eqref{cA} with $P(\E)$ given by \eqref{choiceP}. It is convenient to work with generating functions starting with constant term 1. We first address this normalization, for which we need the following result of Borisov-Libgober \cite[Prop.~3.1]{BL1}. For a vector bundle $E$ on $S$ with Chern roots $x_1, \ldots, x_r$, the following equality holds 
$$
\sfT^{\elg}(E,q,y) = \prod_{i=1}^{r} x_i \frac{\theta_1(q,e^{x_i}y^{-1})}{\theta_1(q,e^{x_i})},
$$
where $\theta_1(q,y)$ denotes the following Jacobi theta function
\begin{align*}
\theta_1(q,y) :=\,& \sum_{n \in \Z} (-1)^n q^{\frac{1}{2}\big(n+\frac{1}{2}\big)^2} y^{n+\frac{1}{2}}\\
=&\,q^{\frac{1}{8}}(y^{\frac{1}{2}}-y^{-\frac{1}{2}})\prod_{n=1}^\infty(1-q^n) (1-yq^n)(1-y^{-1}q^n).
\end{align*}
\begin{definition} \label{defZ}
Define the following functions
\begin{align*}
f^{\cob}(s,\bfv) :=&\, \sfT^{\cob}_{\C^*}(\s^2,\bfv) \\
=&\,  1+\sum_{k=1}^{\infty} (2s)^k v_k, \\
f^{\elg}(s,q,y) :=&\, \sfT^{\elg}_{\C^*}(\s^2,q,y)  \\
=&\, y^{-\frac{1}{2}} \ch^{\C^*}(\cE(\mathfrak{s}^2)) \, \ch^{\C^*}(\Lambda_{-y} \mathfrak{s}^{-2}) \, \td^{\C^*}(\mathfrak{s}^2) \\
=&\, 2s \frac{\theta_1(q,e^{2s}y^{-1})}{\theta_1(q,e^{2s})},
\end{align*}
where $\sfT^{\star}_{\C^*}(\cdot,u^{\star})$ were introduced at the end of the previous section. For any $a \in A^1(S)$ we abbreviate $\chi(a) := \chi(\O(a))$. For any $\star \in \{\cob,\elg\}$, $a_1, c_1 \in A^1(S)$, we define 
\begin{align*} 
\sfZ_{S}^{\star}(a_1,c_1,s,p,u^\star) := \, &(2s)^{-\chi(\O_S)} \Bigg( \frac{2s}{f^\star(s,u^\star)} \Bigg)^{-\chi(c_1-2a_1)} \Bigg( \frac{-2s}{f^\star(-s,u^\star)} \Bigg)^{-\chi(2a_1-c_1)} \\
&\times \sum_{n_1, n_2 \geq 0} p^{n_1+n_2} \int_{S^{[n_1]} \times S^{[n_2]}} \widetilde{\Psi}(a_1,c_1-a_1,n_1,n_2), 
\end{align*}
where $\widetilde{\Psi}$ was defined in the previous section and we take $P(\E) = P^{\star}(\E,u^\star)$ as in \eqref{choiceP}. The coefficient in front has been chosen such that
\begin{align*}
\sfZ_{S}^{\elg}(a_1,c_1,s,p,y,q) &\in 1 + p \, \Q[y^{\pm \frac{1}{2}}](\!(s)\!)[[p,q]], \\
\sfZ_{S}^{\cob}(a_1,c_1,s,p,\bfv) &\in 1 + p \, \Q[\bfv](\!(s)\!)[[p]]. 
\end{align*}
\end{definition}

\begin{proposition} \label{univ}
There exist universal functions 
\begin{align*}
A_1^{\elg}(s,p,y,q), \ldots, A_7^{\elg}(s,p,y,q) &\in 1 + p \, \Q[y^{\pm \frac{1}{2}}](\!(s)\!)[[p,q]], \\
A_1^{\cob}(s,p,\bfv), \ldots, A_7^{\cob}(s,p,\bfv) &\in 1 + p \, \Q[\bfv](\!(s)\!)[[p]],
\end{align*} 
such that for any smooth projective surface $S$, $\star \in \{\cob,\elg\}$, and $a_1, c_1 \in A^1(S)$ we have
$$
\sfZ_{S}^{\star}(a_1,c_1,p,u^\star) = \big(A_1^{\star}\big)^{a_1^2} \, \big(A_2^{\star}\big)^{a_1 c_1} \, \big(A_3^{\star}\big)^{c_1^2} \, \big(A_4^{\star}\big)^{a_1 K_S} \, \big(A_5^{\star}\big)^{c_1 K_S} \, \big(A_6^{\star}\big)^{K_S^2} \, \big(A_7^{\star}\big)^{\chi(\O_S)}.
$$
\end{proposition}
\begin{proof}
Let $\sfT^{\star}_{\C^*}(\cdot,u^\star)$ be the $\C^*$-equivariant version of \eqref{defTcob} and \eqref{defTelg} defined at the end of Section \ref{exprinHilb}. The proof of this proposition is almost verbatim the same as \cite[Prop.~3.3]{GK1}, which in turn makes use of \cite[Lem~5.5]{GNY1}. The main ideas are as follows: \\

\noindent \textbf{Step 1: Universal dependence.} By \cite{EGL}, for any polynomial expression $X$ in Chern classes of
$$
T_{S^{[n]}}, \qquad \O(a_1)^{[n]}, \qquad \O(c_1)^{[n]},
$$
there exists a polynomial $Y$ in the Chern numbers $(a_1^2,a_1 c_1,c_1^2,a_1 K_S, c_1 K_S, K_S^2, \chi(\O_S))$ such that
$$
\int_{S^{[n]}} X = Y,
$$ 
where $Y$ only depends on $X$. The integrals appearing in Mochizuki's formula (Theorem \ref{mocthm} with $P(\E)$ given by \eqref{choiceP}) are over $S^{[n_1]} \times S^{[n_2]}$. Defining $S_2 := S \sqcup S$ (disjoint union), we have
$$
S_2^{[n]} = \bigsqcup_{n_1+n_2 = n} S^{[n_1]} \times S^{[n_2]}
$$
and the integrals appearing in Mochizuki's formula can be expressed as $\int_{S^{[n]}_2} X$, to which \cite{EGL} applies (as shown in \cite[Sect.~5]{GNY1}). In particular, there exist
$$
G^{\elg} \in \Q[y^{\pm \frac{1}{2}},x_1, \ldots, x_7](\!(s)\!)[[p,q]], \qquad G^{\cob} \in \Q[\bfv,x_1, \ldots, x_7](\!(s)\!)[[p]]
$$
such that for all $S,a_1,c_1$ and $\star = \elg, \cob$, we have
$$
\sfZ_{S}^{\star}(a_1,c_1,p,u^\star) = \exp G^{\star}(a_1^2,a_1 c_1,c_1^2,a_1 K_S, c_1 K_S, K_S^2, \chi(\O_S)).
$$

\noindent \textbf{Step 2: Multiplicativity} Let $S',S''$ be any smooth projective surfaces, which are \emph{not necessarily connected}, let $S = S' \sqcup S''$, and let $a_1,c_1 \in A^1(S)$. Denote
$$
a_1' = a_1|_{S'}, \qquad a_1'' = a_1|_{S''}, \qquad c_1' = c_1|_{S'}, \qquad c_1'' = c_1|_{S''}.
$$
We claim that 
\begin{equation*}
\sfZ_{S}^{\star}(a_1,c_1,p,u^\star) = \sfZ_{S'}^{\star}(a_1',c_1',p,u^\star) \, \sfZ_{S''}^{\star}(a_1'',c_1'',p,u^\star). 
\end{equation*}
This follows from the multiplicative properties \eqref{defTcob} and \eqref{defTelg} of $\sfT_{\C^*}^\star(\cdot, u^\star)$, which were defined at the end of Section \ref{exprinHilb}.

Next, we take 7 triples $(S^{(i)},a_1^{(i)},c_1^{(i)})$, where $S^{(i)}$ is an irreducible smooth projective surface and $a_1^{(i)},c_1^{(i)} \in A^1(S^{(i)})$, such that the corresponding vectors $$w_i:=((a_1^{(i)})^2,\ldots, \chi(\O_{S^{(i)}})) \in \Q^7$$ are $\Q$-independent (see Section \ref{toricsec} for one such choice). Then the vector $w = (a_1^2, \ldots, \chi(\O_S))$ of Chern numbers of any triple $(S,a_1,c_1)$, where $S$ is an irreducible smooth projective surface and $a_1, c_1 \in A^1(S)$, can also be realized as the vector of Chern numbers of an appropriate disjoint union of $(S^{(i)},a_1^{(i)},c_1^{(i)})$ (as long as the coefficients $n_i$ in the decomposition $w = \sum_{i=1}^{7} n_i w_i$ are non-negative integers).  From this observation $A_1^\star, \ldots, A_7^\star$ can be constructed in terms of $G^\star$ evaluated on the basis $\{w_i\}_{i=1}^{7}$ as in \cite[Lem.~5.5]{GNY1}.
\end{proof}

For $\star \in \{\cob,\elg\}$ and any $\underline \alpha=(\alpha_1,\alpha_2,\alpha_3,\alpha_4,\alpha_5,\alpha_6,\alpha_7) \in \Z^7$ we define
\begin{align}
\begin{split} \label{defsfA}
\sfA_{\underline \alpha }^{\star}(s,p,u^\star)
:=&- 2 \Bigg( 2^{-1} \Bigg( \frac{2s}{f^\star(s,u^\star)} \Bigg)^{2} \Bigg( \frac{-2s}{f^\star(-s,u^\star)} \Bigg)^{2} p^{-1} A_1(s,2p,u^\star) \Bigg)^{\alpha_1}   \\
&\times \Bigg( 2 \Bigg( \frac{2s}{f^\star(s,u^\star)} \Bigg)^{-2} \Bigg( \frac{-2s}{f^\star(-s,u^\star)} \Bigg)^{-2} p A_2(s,2p,u^\star)  \Bigg)^{\alpha_2}   \\
&\times \Bigg( 2^{-\frac{1}{2}} \Bigg( \frac{2s}{f^\star(s,u^\star)} \Bigg)^{\frac{1}{2}} \Bigg( \frac{-2s}{f^\star(-s,u^\star)} \Bigg)^{\frac{1}{2}} A_3(s,2p,u^\star)  \Bigg)^{\alpha_3} \\
&\times \Bigg(  \Bigg( \frac{2s}{f^\star(s,u^\star)} \Bigg) \Bigg( \frac{-2s}{f^\star(-s,u^\star)} \Bigg)^{-1} A_4(s,2p,u^\star)  \Bigg)^{\alpha_4}    \\
&\times \Bigg( 2^{\frac{1}{2}} \Bigg( \frac{2s}{f^\star(s,u^\star)} \Bigg)^{-\frac{1}{2}} \Bigg( \frac{-2s}{f^\star(-s,u^\star)} \Bigg)^{\frac{1}{2}} A_5(s,2p,u^\star)  \Bigg)^{\alpha_5}  \\
&\times A_6(s,2p,u^\star) ^{\alpha_6}   \\
&\times \Bigg( \frac{s}{2} \Bigg( \frac{2s}{f^\star(s,u^\star)} \Bigg) \Bigg( \frac{-2s}{f^\star(-s,u^\star)} \Bigg) A_7(s,2p,u^\star)  \Bigg)^{\alpha_7}.
\end{split}
\end{align}
Proposition \ref{desc}, Theorem \ref{mocthm}, and Proposition \ref{univ} at once imply the following result.
\begin{corollary} \label{univcor}
Suppose $S$ satisfies $b_1(S)=0$ and $p_g(S) > 0$. Let $H, c_1, c_2$ be chosen such that there exist no rank 2 strictly Gieseker $H$-semistable sheaves with Chern classes $c_1, c_2$ on $S$. Assume furthermore that: 
\begin{itemize}
\item[(i)] $c_2 <  \frac{1}{2} c_{1}(c_1-K_S) + 2\chi(\O_S)$.
\item[(ii)] $p_{\ch} > p_{K_S}$, where $p_{\ch}$ and $p_{K_S}$ are the reduced Hilbert polynomials of $\ch$ and $K_S$.
\item[(iii)] For all SW basic classes $a_1$ satisfying $a_1 H \leq (c_1 -a_1) H$ the inequality is strict. 
\end{itemize}
Then
\begin{align*}
\pi_*[M_{S}^{H}(2,c_1,c_2)]^{\vir}_{\Omega_*}&=\Coeff_{s^0 p^{c_2}}
\Big[ \sum_{{\scriptsize{\begin{array}{c} a_1 \in H^2(S,\Z) \\ a_1 H < (c_1-a_1) H \end{array}}}} \SW(a_1) \, 
\sfA_{(a_1^2,a_1 c_1,c_1^2,a_1 K_S, c_1 K_S, K_S^2, \chi(\O_S))}^{\cob}(s,p,\bfv)\Big], \\
Ell^{\vir}(M_{S}^{H}(2,c_1,c_2))&=\Coeff_{s^0 p^{c_2}}
\Big[ \sum_{{\scriptsize{\begin{array}{c} a_1 \in H^2(S,\Z) \\ a_1 H < (c_1-a_1) H \end{array}}}} \SW(a_1) \, 
\sfA_{(a_1^2,a_1 c_1,c_1^2,a_1 K_S, c_1 K_S, K_S^2, \chi(\O_S))}^{\elg}(s,p,q,y)\Big].
\end{align*}
\end{corollary}

\begin{remark} \label{strongform}
Following Remark \ref{assumpmocthm}, we conjecture that this corollary holds without assuming (ii) and (iii) and with the sum replaced by the sum over all $a_1 \in H^2(S,\Z)$. We refer to this as ``the strong form of Mochizuki's formula''.
\end{remark}

\section{Expression in terms of combinatorics} \label{toricsec}

Consider the following seven choices
\begin{align*}
(S,a_1,c_1) = \, &(\PP^2,\O,\O), (\PP^2,\O(1),\O(1)), (\PP^2,\O,\O(1)), (\PP^2,\O(1),\O(2)), \\
&(\PP^1 \times \PP^1, \O,\O), (\PP^1 \times \PP^1, \O(1,0),\O(1,0)), (\PP^1 \times \PP^1, \O,\O(1,0)).
\end{align*}
Then the corresponding $7 \times 7$ matrix with rows 
$$
(a_1^2,a_1 c_1,c_1^2,a_1 K_S, c_1 K_S, K_S^2, \chi(\O_S))
$$
has full rank. Hence the universal functions $A_1^{\star}, \ldots, A_7^{\star}$, for $\star \in \{\cob,\elg\}$, are entirely determined by $\sfZ_S^{\star}$ for the above seven choices of $(S,a_1,c_1)$. Therefore we want to calculate $\sfZ_S^{\star}$ on toric surfaces. We use Atiyah-Bott localization in order to turn this into a combinatorial problem, which can be implemented on a computer, allowing us to determine the universal functions $A_i^{\star}$ up to certain orders in the formal variables $p,u^\star,s$.

Let $S$ be a toric surface with torus $T = \C^{*2}$ and topological Euler characteristic $e(S)$. Let $\{U_\sigma\}_{\sigma = 1, \ldots, e(S)}$ be the cover of maximal $T$-invariant affine open subsets of $S$. On $U_\sigma$ we use coordinates $x_\sigma, y_\sigma$ such that $T$ acts with characters of weight $v_\sigma, w_\sigma \in \Z^2$
$$
t \cdot (x_\sigma,y_\sigma) = (\chi(v_\sigma)(t) \, x_\sigma, \chi(w_\sigma)(t) \, y_\sigma).
$$
Here $\chi(m) : T \rightarrow \C^*$ denotes the character of weight $m \in \Z^2$. 

Consider the integrals over $S^{[n_1]} \times S^{[n_2]}$ appearing in Definition \ref{defZ}. Let 
$$
\widetilde{T} := T \times \C^*_{\mathbb{M}},
$$ 
where $\C^*_{\mathbb{M}} = \C^*$ denotes the torus coming from master space localization (Remark \ref{masterC}), which acts trivially on $S^{[n_1]} \times S^{[n_2]}$. The action of $T$ on $S^{[n_1]} \times S^{[n_2]}$ is induced from the action of $T$ on $S$. The $T$-fixed locus of $S^{[n_1]} \times S^{[n_2]}$ can be indexed by pairs $(\bslambda,\bsmu)$ with
$$
\bslambda = \{\lambda^{(\sigma)}\}_{\sigma=1, \ldots, e(S)}, \ \bsmu = \{\mu^{(\sigma)}\}_{\sigma=1, \ldots, e(S)},
$$
where $\lambda^{(\sigma)}$, $\mu^{(\sigma)}$ are partitions satisfying
\begin{equation} \label{size}
\sum_\sigma |\lambda^{(\sigma)}| = n_1, \ \sum_\sigma |\mu^{(\sigma)}| = n_2.
\end{equation}
Here $|\lambda| = \sum_{i=1}^{\ell} \lambda_i$ denotes the size of partition $\lambda=(\lambda_1 \geq \cdots \geq \lambda_\ell)$. $\curly{l}$ A partition $\lambda$ corresponds to a monomial ideal of $\C[x,y]$ as follows
$$
I_{Z_\lambda} := (y^{\lambda_1}, xy^{\lambda_2}, \ldots, x^{\ell-1}y^{\lambda_\ell}, x^{\ell}).
$$
For a partition $\lambda^{(\sigma)}$ we denote the subscheme defined by the corresponding monomial ideal in variables $x_\sigma,y_\sigma$ by $Z_{\lambda^{(\sigma)}}$. 

Let $a_1,c_1 \in A^1(S)$. In order to apply localization, we make an arbitrary choice of $T$-equivariant structure on the line bundles $\O(a_1)$, $\O(c_1-a_1)$. For any $T$-equivariant divisor $a$, the restriction $\O(a)|_{U_{\sigma}}$
is trivial with $T$-equivariant structure determined by some character of weight $a_\sigma \in \Z^2$. Consider the following element of $K^{0}_{\widetilde{T}}(S^{[n_1]} \times S^{[n_2]})$
\begin{align*}
E_{n_1,n_2} :=\, &R\Gamma(\O(c_1-2a_1)) \otimes \O \otimes \mathfrak{s}^2 + R\Gamma(\O(2a_1-c_1)) \otimes \O \otimes \mathfrak{s}^{-2} + 2 R\Gamma(\O_S) \otimes \O \\
&- R\pi_* R\hom (\I_1(a_1) \otimes \s^{-1} \oplus \I_2(c_1-a_1) \otimes \s,\I_1(a_1) \otimes \s^{-1} \oplus \I_2(c_1-a_1) \otimes \s). 
\end{align*}
Applying Atiyah-Bott localization, we see that the integral in Definition \ref{defZ} is given by
\begin{align}
\begin{split} \label{toricint}
\int_{S^{[n_1]} \times S^{[n_2]}}\widetilde{\Psi}(a_1,c_1-a_1,n_1,n_2) =\, &\sum_{(\bslambda,\bsmu)}  \prod_\sigma \frac{\Eu(H^0(\O(a_1)|_{Z_{\lambda^{(\sigma)}}}))}{\Eu(T_{Z_{\lambda^{(\sigma)}}})} \\
&\times \frac{\Eu(H^0(\O(c_1-a_1)|_{Z_{\mu^{(\sigma)}}}) \otimes \s^2)}{\Eu(T_{Z_{\mu^{(\sigma)}}})} \\
&\times \frac{\sfT^{\star}_{\widetilde{T}}(E_{n_1,n_2}|_{(Z_{\lambda^{(\sigma)}},Z_{\mu^{(\sigma)}})},u^\star)}{\Eu(E_{n_1,n_2}|_{(Z_{\lambda^{(\sigma)}},Z_{\mu^{(\sigma)}})} - T_{Z_{\lambda^{(\sigma)}}} - T_{Z_{\mu^{(\sigma)}}})},
\end{split}
\end{align}
where $\Eu(\cdot)$ denotes $\widetilde{T}$-equivariant Euler class, $T_{Z}$ denotes the $\C^{*2}$-representation of the tangent space of the Hilbert scheme at $Z \subset \C^2$, the sum is over all $(\bslambda,\bsmu)$ satisfying \eqref{size}, and $\sfT^{\star}_{\widetilde{T}}(\cdot,u^\star)$ is the $\widetilde{T}$-equivariant version of \eqref{defTcob} and \eqref{defTelg} (defined as at the end of Section \ref{exprinHilb}). The calculation is now reduced to the computation of the following elements of the $T$-equivariant $K$-group $K^0_{T}(\pt)$
\begin{align*}
&H^0(\O(a)|_{Z_{\lambda^{(\sigma)}}}), \quad R\Hom_S(\O_{Z_{\lambda^{(\sigma)}}}, \O_{Z_{\lambda^{(\sigma)}}}), \quad R\Hom_S(\O_{Z_{\lambda^{(\sigma)}}}, \O_{Z_{\mu^{(\sigma)}}}(a)),
\end{align*}
for certain $T$-equivariant divisors $a$. The first one is straight-forward
$$
Z_{\lambda^{(\sigma)}} = \sum_{i=0}^{\ell(\lambda^{(\sigma)}) - 1} \sum_{j=0}^{\lambda^{(\sigma)}_{i+1}-1} \chi(v_\sigma)^i \, \chi(w_\sigma)^j.
$$
Multiplying by $\chi(a_\sigma)$ gives $H^0(\O(a)|_{Z_{\lambda^{(\sigma)}}})$. Define 
$$
\overline{\chi(m)} := \chi(-m) = \frac{1}{\chi(m)},
$$
for any $m \in \Z^2$. This defines an involution on $K_0^{T}(\pt)$ by $\Z$-linear extension. 
\begin{proposition}
Let $W,Z \subset S$ be 0-dimensional $T$-invariant subschemes supported on a chart $U_\sigma \subset S$ and let $a$ be a $T$-equivariant divisor on $S$ with weight $a_\sigma \in \Z^2$ on $U_\sigma$. Then we have the following equality in $K_0^{T}(pt)$
$$
R\Hom_S(\O_W, \O_Z(a)) = \chi(a_\sigma) \, \overline{W} Z \frac{(1-\chi(v_\sigma)) (1- \chi(w_\sigma))}{\chi(v_\sigma) \chi(w_\sigma)}.
$$
\end{proposition}
\begin{proof}
This follows from a 2-dimensional version of a calculation in \cite{MNOP}. The argument is given in \cite[Prop.~4.1]{GK1}. The main steps are as follows.

Let $v := v_\sigma$, $w := w_\sigma$, $a := a_\sigma$, and write $U_\sigma = \Spec R$ with $R = \C[x_\sigma,y_\sigma]$. Then
$$
R\Hom_S(\O_W, \O_Z(a)) = R \Hom_{U_\sigma}(\O_{W}, \O_{Z}(a)),
$$
because $W,Z$ are supported on $U_\sigma$. The formula of the proposition follows from
\begin{align*}
&\Gamma(U_\sigma, \O(a)) - R \Hom_{U_\sigma}(I_W,I_Z(a)) = \chi(a) \Big( Z + \frac{\overline{W}}{\chi(v) \chi(w)} - \overline{W} Z \frac{(1-\chi(v)) (1- \chi(w))}{\chi(v) \chi(w)}  \Big)
\end{align*}
by using $I_Z = \O_{U_\sigma} - \O_{Z}$, $I_W = \O_{U_\sigma} - \O_{W}$. This formula can be derived as follows.

Choose $T$-equivariant graded free resolutions
\begin{align*}
0 \rightarrow E_r &\rightarrow \cdots \rightarrow E_0 \rightarrow I_{W} \rightarrow 0, \\
0 \rightarrow F_s &\rightarrow \cdots \rightarrow F_0 \rightarrow I_{Z} \rightarrow 0,
\end{align*}
where
\begin{align*}
E_i = \bigoplus_j R(d_{ij}), \ F_i = \bigoplus_j R(e_{ij}).
\end{align*}
Then we have Poincar\'e polynomials
$$
P_W = \sum_{i,j} (-1)^i \chi(d_{ij}), \ P_Z = \sum_{i,j} (-1)^i \chi(e_{ij}),
$$
and
\begin{align*}
R \Hom_{U_\sigma}(I_W,I_Z(a)) &= \sum_{i,j,k,l} (-1)^{i+k} \Hom(R(d_{ij}),R(a+e_{kl})) \\
&= \sum_{i,j,k,l} (-1)^{i+k} R(a+e_{kl} - d_{ij}) \\
&= \frac{\chi(a) \overline{P}_W P_Z}{(1-\chi(v))(1-\chi(w))}.
\end{align*}
The formula follows by eliminating $P_Z, P_W$ using $W =  \O_{U_\sigma} - I_W$, $Z =  \O_{U_\sigma} - I_Z$.
\end{proof}

This reduces the calculation of \eqref{toricint} to combinatorics, which we implemented in a PARI/GP program. We computed $A_i^{\cob}, A_i^{\elg}$ up to the following orders: 
\begin{itemize}
\item $A_i^{\elg}$ up to order $p^i q^j$, where $i\le 6$ and $j\le 14$.
\item $A_i^{\cob}$ up to order $p^6$ with the specialization $v_i=0$ for $i\ge 6$.
%, and up to order $p^i\bfv^I$ with $i\le 9$ and  $|I|=\sum k i_k\le 4$.
\end{itemize}

\section{Further conjectures and consequences} \label{conseqsec}

In this section we introduce two further conjectures:
\begin{itemize}
\item Conjecture \ref{numconj} is a statement about intersection numbers on Hilbert schemes of points. Together with the strong form of Mochizuki's formula it implies Conjectures \ref{conjell} and \ref{conjcob}. It also implies a generalization of Conjectures \ref{conjell} and \ref{conjcob} to arbitrary blow-ups of surfaces $S$ satisfying $b_1(S)=0$, $p_g(S)>0$, $K_S\neq 0$, and the only Seiberg-Witten basic classes of $S$ are $0$ and $K_S$.
\item Conjecture \ref{generalsurfconj} generalizes Conjectures \ref{conjell} and \ref{conjcob} to arbitrary surfaces $S$ with $b_1(S)=0$ and $p_g(S)>0$. This conjecture is a refinement of (part of) a formula from the physics literature due to Dijkgraaf-Park-Schroers \cite[Eqn.~(6.1), lines 2+3]{DPS}. Conjecture \ref{generalsurfconj} implies a blow-up formula (Proposition \ref{blowupcor}). This can be seen as a (partial!) refinement of the blow-up formula of W.-P.~Li and Z.~Qin \cite{LQ1, LQ2}. Conjecture \ref{generalsurfconj} also implies a formula for surfaces $S$ with $b_1(S) = 0$ and canonical divisor with irreducible reduced connected components (Proposition \ref{propdisconn}). This refines a formula from physics due to Vafa-Witten \cite[Eqn.~(5.45)]{VW}. 
\end{itemize}

In this section we encounter the ratio $F_1^\star(-p,u^\star) / F_1^\star(p,u^\star)$. By definition \eqref{L_a}, we have $\sfL_a(\cdots)|_{(-p,q,y)} = \sfL_a(\cdots)|_{(p,q,y)}$ for any even $a$. Hence the definition of $F_1^{\elg}(p,q,y)$ in Section \ref{notsec} implies
$$
\frac{F_1^{\elg}(-p,q,y)}{F_1^{\elg}(p,q,y)} = \frac{\sfL(2 \phi_{0,\frac{1}{2}})}{\sfL(2 \phi_{0,\frac{1}{2}}) \big|_{(-p,q,y)}}.
$$

Before we continue, we motivate the shape of the formula of Conjecture \ref{conjell}.
\begin{remark} \label{motivate}
Let $S$ be a smooth projective surface with $b_1(S) = 0$, $p_g(S)>0$, and $K_S \neq 0$. Suppose the only Seiberg-Witten basic classes of $S$ are $0$ and $K_S$.
\begin{enumerate}
\item From the results of \cite{GK1}, we expected that the formula for virtual elliptic genera of moduli spaces of Gieseker $H$-stable rank 2 sheaves on $S$ should have strong similarities with the case of Hilbert schemes of points. The Dijkgraaf-Moore-Verlinde-Verlinde formula for elliptic genera of Hilbert schemes of points involves a Borcherds type lift of $\phi_{0,1}$, so we expected to be able to express the generating function of elliptic genera in the rank 2 case in terms of similar quasi-Jacobi forms of index 0, which led us to consider $\phi_{0,\frac{k}{2}}$ introduced in Section \ref{intro}.
\item From the results of \cite{GK1} we also expected the generating function of virtual elliptic genera (and cobordism classes) of moduli spaces of Gieseker $H$-stable rank 2 sheaves on $S$ to be of the form $8 A^{\chi(\O_S)} B^{K_S^2}$, for some universal series $A$, $B$. When ``stable=semistable'', moduli spaces of stable sheaves on a $K3$ surface are deformation equivalent to Hilbert schemes of points of the same dimension \cite{Huy, Yos}. Therefore $A$ is given by the DMVV formula (see also Conjecture \ref{generalsurfconj}, which includes the case $S$ is $K3$.).
\item Similarly, $B$ would then be determined on the blow-up of a $K3$ surface in a point. Matching coefficients for virtual dimension $\leq 4$ led to the explicit form of the formula. Once the prediction was in place, we tested it in many examples, and up to much higher virtual dimension, as will be described in Section \ref{versec}.
\end{enumerate}
\end{remark}

\subsection{Numerical conjecture}

The following conjecture generalizes \cite[Conj.~6.1]{GK1}. 

\begin{conjecture} \label{numconj}
Let $\star \in \{\cob,\elg\}$. Let $\underline{\beta} \in \Z^4$ be such that $\beta_1 \equiv \beta_2 \mod 2$ and $\beta_3 \geq \beta_4 - 3$, and let $(\gamma_1,\gamma_2)\in \Z^2$. Then for all $n < \frac{1}{2}(\beta_1-\beta_2)+2\beta_4$, we have  
\begin{align*}
\Coeff_{s^0 p^{4n-\beta_1-3\beta_4}} \Big[& p^{-\beta_1-3\beta_4} \sfA^{\star}_{(\gamma_1,\gamma_2,\beta_1,\gamma_1,\beta_2,\beta_3,\beta_4)}(s,p^4,u^\star) \\
&+ (-1)^{\beta_4} p^{-\beta_1-3\beta_4} \sfA^{\star}_{(\beta_3-\gamma_1,\beta_2-\gamma_2,\beta_1,\beta_3-\gamma_1,\beta_2,\beta_3,\beta_4)}(s,p^4,u^\star) \Big]
\end{align*}
equals the coefficient of $p^{4n-\beta_1-3\beta_4}$ of 
\begin{align*}
\psi_{\gamma_1, \gamma_2, \beta_3, \beta_4}^{\star}(p,u^\star) := 8 \Bigg( \frac{1}{2} F_0^\star(p,u^\star) \Bigg)^{\beta_4} \Bigg( 2 F_1^\star(p,u^\star) \Bigg)^{\beta_3} (-1)^{\gamma_2} \Bigg( \frac{F_1^\star(-p,u^\star)}{F_1^\star(p,u^\star)} \Bigg)^{\gamma_1}.
\end{align*}
\end{conjecture}

We check this conjecture in various cases in Section \ref{versec}. The first application of this conjecture is the following proposition.
\begin{proposition} 
Assume the strong form of Mochizuki's formula (Remark \ref{strongform}). Conjecture \ref{numconj} for $\star = \cob$ implies Conjecture \ref{conjcob}. Conjecture \ref{numconj} for $\star = \elg$ implies Conjecture \ref{conjell}.
\end{proposition}
\begin{proof}
This is proved in \cite[Prop.~6.3]{GK1}. The idea is as follows.

We only need Conjecture \ref{numconj} for $\gamma_1=\gamma_2=0$. Let $S$ be a smooth projective surface satisfying $b_1(S) = 0$, $p_g(S)>0$, $K_S \neq 0$, and the only Seiberg-Witten basic classes of $S$ are $0$ and $K_S$. Let $H,c_1,c_2$ be chosen such that there are no rank 2 strictly Gieseker $H$-semistable sheaves on $S$ with Chern classes $c_1,c_2$. Take $(\beta_1, \beta_2, \beta_3, \beta_4) = (c_1^2,c_1K_S,K_S^2,\chi(\O_S))$. Then clearly $\beta_1 \equiv \beta_2 \mod 2$. For $S$ minimal of general type, $\beta_3 \geq \beta_4 - 3$ holds by Noether's inequality. For other surfaces, $\beta_3 \geq \beta_4 - 3$ follows from some elementary considerations as explained in the proof of \cite[Prop.~6.3]{GK1}. 

Suppose $c_2$ satisfies
\begin{equation} \label{vdineq}
\vd < c_1^2 - 2c_1K_S + 5 \chi(\O_S),
\end{equation}
where $\vd$ is given by \eqref{vdformula}. Assume Conjecture \ref{numconj} for $\star = \elg$ ($\star = \cob$) holds and the strong form of Mochizuki's formula holds. Then all assumptions of Corollary \ref{univcor} are satisfied and Conjecture \ref{conjell} (Conjecture \ref{conjcob}) follows as long as \eqref{vdineq} is satisfied. When \eqref{vdineq} is not satisfied, we can replace $c_1$ by $c_1+tH$ for $t \gg 0$. Tensoring by $\O_S(tH)$, $M_S^H(r,c_1,c_2)$ is isomorphic to a moduli space of the same virtual dimension for which \eqref{vdineq} is satisfied.
\end{proof}

\subsection{Fixed first Chern class}

Let $S$ be a smooth projective surface with $b_1(S)=0$ and polarization $H$. Let $c_1$ be chosen such that there exist no rank 2 Gieseker $H$-semistable sheaves with first Chern class $c_1$. We consider the generating functions
\begin{align*}
\sfZ_{S,c_1}^{\cob}(p,\bfv) &= \sum_{c_2} \pi_* [M_S^H(2,c_1,c_2)]^{\vir}_{\Omega_*} \, p^{\vd(M_S^H(2,c_1,c_2))}, \\
\sfZ_{S,c_1}^{\elg}(p,q,y) &= \sum_{c_2} Ell^{\vir}(M_S^H(2,c_1,c_2)) \, p^{\vd(M_S^H(2,c_1,c_2))}.
\end{align*}
Set $i:=\sqrt{-1}$. We make use of the following general principle. Let $\psi(x)  = \sum_{n=0}^{\infty} \psi_n x^n$ be any formal power series in $x$ and suppose we want to extract the coefficients $\psi_n$ for which $n \equiv \alpha \mod 4$ for some $\alpha \in \Z$. This can be done as follows:
\begin{align}
\begin{split} \label{principle}
\sum_{k=0}^{3} \frac{i^{-\alpha k}}{4} \psi(i^k x) &= \sum_{k=0}^{3} \sum_{n=0}^{\infty} \frac{i^{k(n - \alpha)}}{4} \psi_n x^n \\
&=\sum_{n=0}^{\infty} \Big( \frac{1}{4} \sum_{k=0}^3 i^{k(n - \alpha)} \Big) \psi_n x^n \\
&=\sum_{n \equiv \alpha \mod 4} \psi_n x^n.
\end{split}
\end{align}
From this simple principle, the following two propositions follow at once (their analogs for virtual $\chi_y$-genus are \cite[Prop.~6.4, 6.5]{GK1}).
\begin{proposition} 
Let $\star \in \{\cob,\elg\}$ and assume Conjecture \ref{numconj} is true for $\star$. Let $\underline{\beta} \in \Z^4$ such that $\beta_1 \equiv \beta_2 \mod 2$ and $\beta_3 \geq \beta_4 - 3$, and let $(\gamma_1,\gamma_2)\in \Z^2$. 
\begin{align*}
&\Coeff_{s^0} \Big[ p^{-\beta_1-3\beta_4} \sfA^{\star}_{(\gamma_1,\gamma_2,\beta_1,\gamma_1,\beta_2,\beta_3,\beta_4)}(s,p^4,u^\star) \\
&\qquad\quad \ \, + (-1)^{\beta_4} p^{-\beta_1-3\beta_4} \sfA^{\star}_{(\beta_3-\gamma_1,\beta_2-\gamma_2,\beta_1,\beta_3-\gamma_1,\beta_2,\beta_3,\beta_4)}(s,p^4,u^\star) \Big]\\
&=2(-1)^{\gamma_2}\sum_{k=0}^3 (i^k)^{\beta_1-\beta_4} \Bigg(\frac{1}{2} F_0^\star(i^kp,u^\star) \Bigg)^{\beta_4}\Bigg( 2 F_1^\star(i^kp,u^\star) \Bigg)^{\beta_3}\Bigg(\frac{F_1^\star(-i^kp,u^\star)}{F_1^\star(i^kp,u^\star)} \Bigg)^{\gamma_1}+O(p^{\beta_1-2\beta_2+5\beta_4}).
\end{align*}
\end{proposition}

\begin{proposition} \label{fixedc1prop}
Let $\star \in \{\cob,\elg\}$. Assume Conjecture \ref{conjell} is true when $\star = \elg$ and Conjecture \ref{conjcob} is true when $\star = \cob$. Let $S$ be a smooth projective surface with $b_1(S) = 0$, $p_g(S)>0$, and $K_S \neq 0$. Suppose the Seiberg-Witten basic classes of $S$ are $0$ and $K_S$. Let $H, c_1$ be chosen such that there are no rank 2 strictly Gieseker $H$-semistable sheaves on $S$ with first Chern class $c_1$. Then
\begin{align*} 
\sfZ_{S,c_1}^\star(p,u^\star) = 2\sum_{k=0}^3 (i^k)^{c_1^2 - \chi(\O_S)} \Bigg( \frac{1}{2} F_0^{\star}(i^k p,u^\star) \Bigg)^{\chi(\O_S)} \Bigg( 2 F_1^{\star}(i^k p,u^\star) \Bigg)^{K_S^2}.
\end{align*}
\end{proposition}

In fact Conjecture \ref{numconj} can be used to generalize this proposition as follows (the analog for virtual $\chi_y$-genus is \cite[Prop.~6.6]{GK1}). 
\begin{proposition} \label{nonmingt}
Assume the strong form of Mochizuki's formula holds (Remark \ref{strongform}). Let $\star \in \{\cob,\elg\}$ and assume Conjecture \ref{numconj} is true for $\star$.
Let $S_0$ be a smooth projective surface with $b_1(S_0) = 0$, $p_g(S_0)>0$, and $K_{S_0} \neq 0$. Suppose the Seiberg-Witten basic classes of $S_0$ are $0$ and $K_{S_0}$. Suppose $S$ is obtained from $S_0$ by repeated blow-ups and denote the total transforms of the exceptional divisors by $E_1, \ldots, E_m$. Suppose that $K_S^2 \ge \chi(\O_S)-3$. Let $H, c_1$ be chosen such that there exist no rank 2 strictly Gieseker $H$-semistable sheaves on $S$ with first Chern class $c_1$. Then
\begin{align*}
&\sfZ_{S,c_1}^\star(p,u^\star) = \\
&2\sum_{k=0}^3 (i^k)^{c_1^2-\chi(\O_S)}\Bigg(\frac{1}{2} F^{\star}_0(i^k p,u^\star) \Bigg)^{\chi(\O_S)}\Bigg(2 F^{\star}_1(i^k p,u^\star) \Bigg)^{K_{S}^2} \prod_{j=1}^m\Bigg(1+(-1)^{c_1E_j} \frac{F^{\star}_1(i^k p,u^\star)}{F^{\star}_1(-i^k p,u^\star)} \Bigg).
\end{align*}
\end{proposition}
\begin{proof}
Let $M:=\{1, \ldots, m\}$ and write $E_I = \sum_{i \in I} E_i$ for any $I \subset M$. Then $K_S = K_{S_0} + E_M$ and $\chi(\O_S) = \chi(\O_{S_0})$. The SW basic classes of $S$ are $E_I$ (with SW invariant 1) and $K_{S_0} + E_I = K_S - E_{M \setminus I}$ (with SW invariant $(-1)^{\chi(\O_S)}$), where $I$ runs over all subsets of $M$ \cite[Thm.~7.4.6]{Mor}. The rest is a short calculation (as for \cite[Prop.~6.6]{GK1}).
\end{proof}

\subsection{Arbitrary surfaces with holomorphic 2-form}

The following conjecture generalizes \cite[Conj.~6.7]{GK1}. This conjecture can be seen as a refinement of (part of) a formula of Dijkgraaf-Park-Schroers \cite[Eqn.~(6.1), lines 2+3]{DPS}.
\begin{conjecture}\label{generalsurfconj}
Let $\star \in \{\cob,\elg\}$ and let $S$ be a smooth projective surface with $b_1(S) = 0$ and $p_g(S)>0$. Let $H,c_1,c_2$ be chosen such that there are no rank 2 strictly Gieseker $H$-semistable sheaves on $S$ with first Chern class $c_1$. For $M:=M_S^H(2,c_1,c_2)$, the coefficient of $p^{\vd(M)}$ of $\sfZ^\star_{S,c_1}(p,u^\star)$ equals the coefficient of $p^{\vd(M)}$ of
\begin{align*}
\psi_{S,c_1}^\star(p,u^\star) := 4 \Bigg(\frac{1}{2}F_0^{\star}(p,u^\star)\Bigg)^{\chi(\O_S)}\Bigg(2F_1^{\star}(p,u^\star)\Bigg)^{K_{S}^2}  \sum_{a \in H^2(S,\Z)} \SW(a)(-1)^{c_1 a} \Bigg(\frac{F_1^{\star}(-p,u^\star)}{F_1^{\star}(p,u^\star)}\Bigg)^{a K_S}.
\end{align*}
\end{conjecture}
If there are no strictly Gieseker $H$-semistable sheaves with first Chern class $c_1$, this conjecture implies (using \eqref{principle})
$$
\sfZ_{S,c_1}^{\star}(p,u^\star) = \frac{1}{2} \psi_{S,c_1}^{\star}(p,u^\star) + \frac{1}{2} i^{c_1^2 - \chi(\O_S)} \psi_{S,c_1}^{\star}(ip,u^\star).
$$
\begin{remark}
A simple computation shows that this conjecture implies both Propositions \ref{fixedc1prop} (without assuming Conjectures \ref{conjell}, \ref{conjcob}) and \ref{nonmingt} (without assuming Conjecture \ref{numconj} and without assuming $\chi(\O_S) \geq K_S^2-3$). In fact, this conjecture implies both Conjectures \ref{conjell} and \ref{conjcob}.
\end{remark}

The first application of Conjecture \ref{generalsurfconj} is the following blow-up formula. The analog for virtual $\chi_y$-genus is \cite[Prop.~6.9]{GK1}. The proof follows immediately from the description of the Seiberg-Witten basic classes and invariants of a blow-up \cite[Thm.~7.4.6]{Mor}.
\begin{proposition} \label{blowupcor}
Let $\star \in \{\cob,\elg\}$. Assume Conjecture \ref{generalsurfconj} holds for $\star$. Let $\pi : \widetilde{S} \rightarrow S$ be the blow-up in a point of a  smooth projective surface $S$ with $b_1(S) = 0$, $p_g(S)>0$. Suppose $H, c_1$ are chosen such that there are no rank 2 strictly Gieseker $H$-semistable sheaves on $S$ with first Chern class $c_1$. Let $\widetilde{c}_1 = \pi^* c_1 - \epsilon E$ with $\epsilon=0,1$ and suppose $\widetilde{H}$ is a polarization on $\widetilde{S}$ such that there are no rank 2 strictly Gieseker $\widetilde{H}$-semistable sheaves on $\widetilde{S}$ with first Chern class $\widetilde{c}_1$. Then
\begin{align}
\begin{split} \label{blowupform}
\sfZ_{\widetilde S,\widetilde c_1}^{\star}(p,u^\star)&=\frac{1}{2} \psi^\star_{\widetilde S,\widetilde c_1}(p,u^\star)+ \frac{1}{2} i^{\widetilde{c}_1^2-\chi(\O_{\widetilde{S}})} \psi^\star_{\widetilde S,\widetilde c_1}\big(ip,u^\star\big), \\
\psi^\star_{\widetilde S,\widetilde c_1}(p,u^\star)&=\frac{1}{2}\Big(F_1^\star(p,u^\star)^{-1} +(-1)^\epsilon F_1^\star(-p,u^\star)^{-1}\Big) \psi^\star_{S,c_1}(p,u^\star).
\end{split}
\end{align}
\end{proposition}

 In \cite{LQ1, LQ2}, Li-Qin derive a formula for the virtual Hodge polynomials of a blow-up. Here ``virtual'' is meant in the sense of Deligne's weight filtration, not virtual classes. Interestingly the $\chi_{y}^{\vir}$-specialization of \eqref{blowupform} coincides with the $\chi_y$-specialization of the virtual Hodge polynomials of Li-Qin \cite[Prop.~6.9]{GK1}.

The second (more involved) application of Conjecture \ref{generalsurfconj} is to surfaces with $b_1=0$, $p_g>0$, and canonical divisor with irreducible reduced connected components. The following proposition is proved in the same way as \cite[Prop.~6.11]{GK1} and refines a formula of Vafa-Witten \cite[Eqn.~(5.45)]{VW}. 
\begin{proposition} \label{propdisconn}
Let $\star \in \{\cob,\elg\}$. Let $S$ be a smooth projective surface with $b_1(S) = 0$ and $p_g(S)>0$. Suppose $|K_S|$ contains a reduced curve whose connected components $C_1, \ldots, C_m$ are irreducible. Let $N_{C_j/S}$ denote the normal bundles of $C_j \subset S$. Let $H,c_1$ be chosen such that there are no rank 2 strictly Gieseker $H$-semistable sheaves with first Chern class $c_1$. Then 
\begin{align*} 
&\sfZ_{S,c_1}^\star(p,u^\star) = 2 \Bigg(\frac{1}{2} F_0^\star(p,u^\star) \Bigg)^{\chi(\O_S)} \prod_{j=1}^{m} \Bigg( (2F^\star_1(p,u^\star))^{C_j^2} + (-1)^{c_1 C_j + h^0(N_{C_j/S})} (2F^\star_1(-p,u^\star))^{C_j^2} \Bigg) \\
&+ 2(-i)^{c_1^2 - \chi(\O_S)} \Bigg(\frac{1}{2} F_0^\star(-ip,u^\star)\Bigg)^{\chi(\O_S)}  \prod_{j=1}^{m} \Bigg( (2F^\star_1(-ip,u^\star))^{C_j^2} + (-1)^{c_1 C_j + h^0(N_{C_j/S})} (2F^\star_1(ip,u^\star))^{C_j^2} \Bigg).
\end{align*}
\end{proposition}
\begin{proof}
The Seiberg-Witten basic classes and invariants of $S$ can be described as follows \cite[Lem.~6.14]{GK1}. For  $I \subset M:=\{1, \ldots,m\}$, let $C_I:=\sum_{i \in I} C_i$, where $I = \varnothing$ corresponds to the zero divisor. Next, define an equivalence relation $I \sim J$, when $C_I$ and $C_J$ are linearly equivalent. Then the SW basic classes of $S$ are $\{C_I \in H^2(S,\Z)\}_{I \subset M}$ and
$$
\SW(C_I) = |[I]| \prod_{i \in I} (-1)^{h^0(N_{C_i/S})},
$$ 
where $|[I]|$ is the number of elements of the equivalence class of $I$ and $N_{C_i / S}$ is the normal bundle of $C_i$. The rest of the proof is an easy calculation as for \cite[Prop.~6.11]{GK1}.
\end{proof}

\section{Verification of the conjectures in examples} \label{versec}

In this section we use Corollary \ref{univcor} in order to verify Conjectures \ref{conjell}, \ref{conjcob}, \ref{numconj}, \ref{generalsurfconj} in a number of examples. In Section \ref{toricsec} we mentioned that we have determined the universal functions $A_i^{\star}$ up to the following orders: 
\begin{itemize}
\item $A_i^{\elg}$ up to order $p^i q^j$, where $i\le 6$ and $j\le 14$.
\item $A_i^{\cob}$ up to order $p^6$ with the specialization $v_i=0$ for $i\ge 6$.
\end{itemize}
Using the methods of \cite{EGL} we have determined the cobordism class of $K3^{[n]}$ for $n\le 7$. This determines 
$$
F_0^{\cob}(p,\bfv) \mod p^{16}.
$$ 
We use this in the verifications of the conjectures for $\sfZ_{S,c_1}^{\cob}(p,\bf{v})$. Assuming Conjecture \ref{conjcob} in a special case, see Remark \ref{K3bl}, we determine 
$$
F_1^{\cob}(p,\bfv) = L(p,\bfv)|_{v_6=v_7=\cdots=0} \mod p^{14}.
$$
We use this as our definition of $L(p,\bfv)$ in our verifications of the conjectures for $\sfZ_{S,c_1}^{\cob}(p,\bf{v})$.

\subsection{$K3$ surfaces and their blow-ups}

For a $K3$ surface $S$ the only Seiberg-Witten basic class is $0$. Let $H, c_1$ be chosen such that $c_1H>0$ is odd. We put $c:=c_1^2\in 2\Z$. For $2\le c\le 14$ even and $\star = \cob, \elg$, we determined
$$
\sfZ_{S,c_1}^\star(p,u^\star) \mod p^{\min\{c+10,22-c\}}.
$$ 
The orders of $u^\star$ are determined by the orders up to which we calculated $A_i^\star$ as mentioned at the beginning of this section. We suppress these orders throughout this section. Conjecture \ref{generalsurfconj} is confirmed in this range.

For $S,H,c_1$ as above, let $\pi:\widetilde S\to S$ be the blow-up of $S$ in a point and denote the exceptional divisor by $E$. Let $\widetilde c_1=\pi^*c_1+\epsilon E$, and $\widetilde H=rH-E$ with $r\gg 0$ and $r+\epsilon$
odd. For $\star  = \cob,\elg$, $\epsilon=0,1$, and $2\le c\le 14$ even, we determined
$$
\sfZ_{\widetilde S,\widetilde c_1}^{\star}(p,u^\star) \mod p^{\min\{c+2\epsilon+10,22-c+\epsilon\}}.
$$
Conjectures \ref{conjell} and \ref{conjcob} are verified in this range.

\begin{remark} \label{K3bl} 
\hfill
\begin{enumerate}
\item Let $\widetilde{S}$ be the blow-up of an elliptic $K3$ surface with section and 24 nodal singular fibres (and no further singular fibres). Denote the (pull-back of) a section and fibre class on $K3$ by $B, F$ respectively.  As before we denote the exceptional divisor by $E$. Assume Conjecture \ref{conjcob} is true for $\widetilde{S}$, $\widetilde{H}$, and $\widetilde{c}_1=B, B+F,B+E, B+F+E$. 
Then\footnote{In the sum on RHS, for each of the choices of $\widetilde{c}_1$, we choose a possibly different polarization $\widetilde{H}$ such that there are no rank 2 strictly Gieseker $\widetilde{H}$-semistable sheaves on $\widetilde{S}$ with first Chern class $\widetilde{c}_1$.}
$$L(p,\bfv)^{-1}=\frac{1}{F_0^{\cob}(p,\bfv)^2}\sum_{\widetilde{c}_1=B, B+F,B+E, B+F+E} \sum_{\widetilde{c}_2} \pi_* [M^{\widetilde{H}}_{\widetilde{S}}(2,\widetilde{c}_1,\widetilde{c}_2)]^{\vir}_{\Omega_*} \, p^{4\widetilde{c}_2-\widetilde{c}_1^2-6}.$$ 

\item Therefore, using the above notation and assuming Conjecture \ref{conjcob} for $\widetilde{S}$ with $c=4,6$, $\epsilon=0,1$, we determine the  $L(p,\bfv)|_{v_6=v_7=\ldots=0}$ modulo $p^{14}$.
In particular we find
\begin{align*}
L(p,\bfv)^{-1}&=1 + 2v_1p - 16v_3p^3+ 4(v_1^4-3v_2v_1^2 + v_3v_1)p^4\\&\quad + 4(v_1^5- 6v_1^3v_2-12v_1^2v_3+9v_1v_2^2 + 22 v_2v_3   + 38v_5 )p^5+O(p^6).
\end{align*}
\end{enumerate}
\end{remark}

\subsection{Elliptic surfaces}

Let $S\to \PP^1$ be an elliptic surface with a section $B$, $12n$ rational nodal fibres, and no other singular fibres. We take $n \geq 2$.
The canonical class is $K_S=(n-2)F$, where $F$ denotes the class of a fibre. Then $\chi(\O_S)=n$ and $B^2=-n$. Moreover, the Seiberg-Witten basic classes of $S$ are $0,F, \ldots, (n-2)F$ and
$$
\SW(pF) = (-1)^p \binom{n-2}{p}.
$$
This follows from the description of SW basic classes and invariants in the proof of Proposition \ref{propdisconn}. For polarizations $H$ such that $c_1H > 2K_SH$ is odd, we determined
$$
\sfZ_{S,\epsilon B+d F}^\star(p,u^\star) \mod p^{\min\{28-c_1^2-3n,5n+c_1^2-2\epsilon(n-2)\}},
$$
for all $n=3,\ldots,6$, $\epsilon=0,1$, $d=0,\ldots,8$, and $\star = \cob, \elg$. This allows us to verify Conjecture \ref{generalsurfconj} in this range.

\subsection{Double covers}

We consider double covers $\pi:S\to \PP^2$ branched along a smooth curve of degree $8$.
Then $K_S^2=2$, $\chi(\O_S)=4$ and we note that $|K_S|$ contains smooth connected curves.
Hence the Seiberg-Witten basic classes are $0$, $K_S$ with Seiberg-Witten invariants $\SW(0)=1$, $\SW(K_S)=(-1)^{\chi(\O_S)}=1$. In this section we assume that the strong form of Mochizuki's formula holds (Remark \ref{strongform}).

We denote by $L$ the pullback of the hyperplane class on $\PP^2$. We assume for simplicity that 
$\Pic(S)=\Z L$ and take polarization $H=L$. Then there are no rank 2 strictly $\mu$-semistable sheaves on $S$ with first Chern class $c_1=L$. For both $\star = \cob, \elg$, we determined
\begin{align*}
\sfZ_{S,L}^\star(p,u^{\star}) \mod p^{14}, \\
\sfZ_{S,2L,\odd}^\star(p,u^\star) \mod p^8,
\end{align*}
where ``$\odd$'' means that we only sum over $c_2$ odd, so that there are no rank 2 strictly Gieseker semistable sheaves on $S$ with Chern classes $c_1=2L$ and $c_2$. We verified Conjectures \ref{conjell} and \ref{conjcob} in this range. In particular, for  $M=M^L_S(2,L,4)$ we find
$$
c^\vir_1(M)^2=48, \, c^\vir_2(M)=120.
$$

Next we consider double covers of $\pi : S \rightarrow \PP^1\times\PP^1$ branched along a smooth curve of bidegree $(6,6)$ and $(6,8)$.
%\begin{LG} For the cobordism class it is only $(6,6)$. I did the computation for the cobordism class also also for $(6,8)$ and some others, but they do not yield any coefficient which is not trivially zero.
%\end{LG}
Denote the pull-backs of the classes of $\PP^1 \times \{\mathrm{pt}\}$ and $\{\mathrm{pt}\} \times \PP^1$ by respectively $\widetilde{B}$ and $\widetilde{F}$, and let $c_1=\epsilon_1 \widetilde{B}+\epsilon_2 \widetilde{F}$. For a suitably chosen polarization $H$, $\star = \cob, \elg$, $\epsilon_1=0,1$, and $1 \leq \epsilon_2\le 5$, we determined
\begin{align*}
&\sfZ_{S,c_1}^\star(p,u^\star) \mod p^{\min\big\{25-4(\epsilon_1+\epsilon_2-\epsilon_1\epsilon_2),13-4\epsilon_1\epsilon_2,-3+4(2\epsilon_1+2\epsilon_2-\epsilon_1\epsilon_2)\big\}}, \quad \mathrm{for \ bidegree \ } (6,6)  \\
&\sfZ_{S,c_1}^\star(p,u^\star) \mod p^{\min\big\{35-4(2\epsilon_1+\epsilon_2-\epsilon_1\epsilon_2),7-4\epsilon_1\epsilon_2,-25+4(4\epsilon_1+2\epsilon_2-\epsilon_1\epsilon_2)\big\}}, \quad \mathrm{for \ bidegree \ } (6,8).  
\end{align*}
We verified Conjectures \ref{conjell} and \ref{conjcob} in this range. E.g.~for bidegree $(6,8)$ and $M:=M^H_S(2,\widetilde{B},6)$ we find 
$$
c_3^\vir(M)=-36864, \, c_1^\vir(M)c_2^\vir(M)= -67584, \, c_1^\vir(M)^3=- 90112.
$$
This is in accordance with Conjecture \ref{conjcob}.

Denote by $\FF_1 = \PP(\O_{\PP^1} \oplus \O_{\PP^1}(1))$ the first Hirzebruch surface. Let $B$ be the section satisfying $B^2=-1$ and let $F$ be the class of the fibre $\FF_1 \rightarrow \PP^1$. We consider double covers $\pi : S \rightarrow \FF_1$ branched over a smooth connected curve in
$$
|\O_{\FF_1}(6B + 10F)|.
$$
Denote the pull-backs of $B,F$ by $\widetilde{B}, \widetilde{F}$ respectively and let $c_1=\epsilon_1 \widetilde{B}+\epsilon_2 \widetilde{F}$. For a suitably chosen polarization $H$, $\star = \cob, \elg$, $-2\leq \epsilon_1\le 6$, and $-2\leq \epsilon_2\le 6$, we determined $\sfZ_{S,c_1}^\star(p,u^\star)$ modulo $p^{\min\{N+28,M\}}$, where $M,N$ are the explicit expressions given in \cite[Sect.~7.4]{GK1} which we do not reproduce here. We verified Conjectures \ref{conjell} and \ref{conjcob} in this range. E.g.~for a suitably chosen polarization $H$ and $M:=M^H_S(-\widetilde{B}+3\widetilde{F},2)$, we find
\begin{align*}
&c_4^\vir(M)=85920, \, c_1^\vir(M) c_3^\vir(M)=161088, \, c_2^\vir(M)^2=241056, \\
&c^\vir_1(M)^2c_2^\vir(M)=279936, \, c_1^\vir(M)^4=345600.
\end{align*}
This is in accordance with Conjecture \ref{conjcob}.

\subsection{Hypersurfaces}

Finally we consider the very general quintic in $S \subset \PP^3$. Then $\Pic(S)$ is generated by the hyperplane class $H$ and $K_S^2 = \chi(\O_S) = 5$. Since $|K_S|$ contains smooth connected curves, the Seiberg-Witten basic classes are $0$, $K_S$ and $\SW(0) = 1$, $\SW(K_S) = (-1)^{\chi(\O_S)} = -1$. We assume that the strong form of Mochizuki's formula holds (Remark \ref{strongform}). 
For $\star = \cob, \elg$, we determined
\begin{align*}
\sfZ_{S,H}^\star(p,u^\star) \mod p^{8}.
\end{align*}
Our answers agree with Conjectures \ref{conjell} and \ref{conjcob}. E.g.~for $M:=M^H_{S}(2,H,6)$ we find that
\begin{align*}
&c^\vir_4(M)=52720, \, c^\vir_1(M)c^\vir_3(M)=93280, \, c_2(M)^2=145200, \, c_1^\vir(M)^2c_2^\vir(M)=157760, \\ 
&c_1^\vir(M)^4=185600.
\end{align*}
This is in accordance with Conjecture \ref{conjcob}.

\begin{remark}
It is remarkable that in all the examples that we computed, the Chern numbers $c_{i_1}^\vir(M)\cdots c_{i_k}^\vir(M)$ have the sign $(-1)^{\vd(M)}$.
This is reminiscent of  \cite[Rem.~5.5]{EGL}, where it is noted that at least for $n\le 7$ all Chern numbers of $S^{[n]}$ are polynomials in $c_1(S)^2$ and $c_2(S)$ with positive coefficients.
\end{remark}

\subsection{Verification of Conjecture \ref{numconj}}

For both $\star \in \{\cob,\elg\}$, we checked Conjecture \ref{numconj} modulo
$$
\mod p^{\min\Big\{2\beta_1-2\beta_2+8\beta_4, 28-4\gamma_1+4\gamma_2,28+4\beta_2-4\beta_3+4\gamma_1-4\gamma_2\Big\}-\beta_1-3\beta_4}
$$
and for $\gamma_1=0,\ldots, 4$, $\gamma_2=0,\ldots,4$, $\beta_1=\beta_2=4,\ldots,10$, $\beta_3=2,\ldots,5$, $\beta_4=2,\ldots,5$. 
The above bounds come from: (1) the bound $\mod p^{\beta_1-2\beta_2+5\beta_4}$ which is part of the statement of Conjecture \ref{numconj} and (2) the order to which we calculated $A_i^\star$ as stated at the beginning of this section (see also \eqref{defsfA}). This allows us to calculate up to $S^{[n_1]} \times S^{[n_2]}$ with $n_1+n_2 \leq 6$. As throughout this section, for $\star=\cob$ we used the specialization $v_i=0$ for $i\ge 6$.

%\nocite{*}

{\tt{gottsche@ictp.trieste.it, m.kool1@uu.nl}}
\end{document}